\documentclass[onefignum,onetabnum]{siamart190516}
\usepackage[utf8]{inputenc}
\usepackage{url}
\usepackage[algo2e,linesnumbered,ruled,vlined,boxed,noresetcount,algosection]{algorithm2e}
\usepackage{booktabs}
\usepackage{multirow}
\usepackage{graphicx}
\usepackage{subfig}
\usepackage{xcolor}
\usepackage{tikz}
\usepackage{pgfplots}
\usepackage{amsmath}
\usepackage{amsfonts}

\DeclareMathOperator{\Ran}{Ran}
\DeclareMathOperator{\linspan}{span}
\DeclareMathOperator{\sign}{sign}

\newcommand{\bdf}{\mathbf{f}}
\newcommand{\bdg}{\mathbf{g}}
\newcommand{\bdG}{\mathbf{G}}
\newcommand{\bdL}{\mathbf{L}}
\newcommand{\bdn}{\mathbf{n}}
\newcommand{\bdN}{\mathbf{N}}
\newcommand{\bdu}{\mathbf{u}}
\newcommand{\bdv}{\mathbf{v}}
\newcommand{\bdx}{\mathbf{x}}
\newcommand{\R}{\mathbb{R}}

\newcommand{\Ntotal}{N_{\mathrm{total}}}

\newcommand{\equispaced}{equi-spaced}

\definecolor{cLA}{HTML}{000000}
\definecolor{cZE}{HTML}{00BB86}
\definecolor{cQR}{HTML}{005CB3}
\definecolor{cCL}{HTML}{A950AC}
\definecolor{cE2}{HTML}{EF4E79}
\definecolor{cE3}{HTML}{F97E38}
\definecolor{cE4}{HTML}{CCBB00}

\headers{Initial Guesses for Sequences of Linear Systems}{A.~P.\ Austin, N.\ Chalmers, and T.\ Warburton}

\title{Initial Guesses for Sequences of Linear Systems in a
GPU-Accelerated Incompressible Flow Solver
\thanks{Submitted to the editors September 25, 2020.
\funding{This research was supported in part by the Exascale Computing Project
(17-SC-20-SC), a collaborative effort of the U.S. Department of Energy Office
of Science and the National Nuclear Security Administration.  This research was
also supported in part by the John K. Costain Faculty Chair in Science at
Virginia Tech.  The first author additionally acknowledges support through the
Research Initiation Program at the Naval Postgraduate School.}}}
\author{Anthony P. Austin\thanks{Dept.\ of Applied Mathematics, Naval Postgraduate School, 833 Dyer Rd., Monterey, CA 93940} \and Noel Chalmers\thanks{AMD Research, Advanced Micro Devices Inc., 7171 Southwest Pkwy., Austin, TX 78735} \and Tim Warburton\thanks{Department of Mathematics, Virginia Tech, 225 Stanger St., Blacksburg, VA 24061}}
\date{August 2020}

\begin{document}

\maketitle

\begin{abstract}
We consider several methods for generating initial guesses when iteratively
solving sequences of linear systems, showing that they can be implemented
efficiently in GPU-accelerated PDE solvers, specifically solvers for
incompressible flow.  We propose new initial guess methods based on stabilized
polynomial extrapolation and compare them to the projection method of Fischer
\cite{Fis1998}, showing that they are generally competitive with projection
schemes despite requiring only half the storage and performing considerably
less data movement and communication.  Our implementations of these algorithms
are freely available as part of the \texttt{libParanumal} collection of
GPU-accelerated flow solvers.
\end{abstract}

\begin{keywords}
initial guesses, iterative solvers, GPU-acceleration, partial differential equations, incompressible flow, projection, extrapolation, least-squares
\end{keywords}

\begin{AMS}
65F10, 65M22
\end{AMS}

\section{Introduction}
\label{SEC:Introduction}

Solving a time-dependent partial differential equation (PDE), discretized via
the method of lines using an implicit (or partially implicit) time-stepping
scheme boils down to solving a sequence of linear systems, one (sometimes more)
at each time step.  For large-scale simulations involving discretizations with
millions (or even billions) of degrees of freedom, these systems are often too
large even to form explicitly, much less solve directly.  Thus, one turns to
iterative methods \cite{Gre1997,Saa2003}.  In particular, Krylov subspace
iterations, which include the popular conjugate gradient (CG) \cite{HS1952} and
generalized minimal residual (GMRES) \cite{SS1986} algorithms, have proven
successful in this area.

With Krylov methods, the computational expense of solving a linear system is
simply the product of the number of iterations required and the cost of each
iteration.  Used as-is, these methods are prone to slow convergence and
stagnation; a good preconditioner is essential for keeping iteration counts
manageable.  Even though preconditioning raises the cost per iteration, the
substantial reduction in iterations it yields typically makes the trade-off
favorable.

Another way to reduce the number of iterations is to supply the method with a
good initial guess for the solution.  Good initial guesses are not available in
all contexts, but in a time-dependent PDE simulation, it is natural to use the
solution computed at the previous time step as a guess for the solution at the
current one.  This requires no effort to implement and can be effective if the
solution to the PDE varies smoothly in time.

In a 1998 paper, Fischer \cite{Fis1998} showed that one can form an even better
initial guess by combining the solutions at several previous time steps instead
of just one.  Specifically, his method takes as the initial guess the solution
to the system whose right-hand side is the orthogonal projection of the
right-hand side of the linear system at the current time step onto the span of
the right-hand sides of the systems at a given number of prior time steps.
This is equivalent to taking the linear combination of the prior solutions that
minimizes the norm of the residual when substituted into the linear system at
the current time step.  To avoid growth in the storage required as the
simulation runs, Fischer periodically restarts the method by discarding the
entire history space of stored solutions and right-hand side vectors.  In
\cite{Chr2017}, Christensen showed how to avoid this ``hard restart'' by using
a classic technique for updating QR factorizations.  This technique was also
employed by al Sayed Ali and Sadkane in \cite{AS2012}, who compared variants of
this projection scheme with initial guess generation methods based on explicit
ODE integrators.  Fischer's principle application of the projection technique
was to incompressible flow, but similar (and in some cases identical) ideas
have emerged in other contexts, including nuclear physics \cite{BILO1997} and
computational electromagnetics \cite{CWSW2004,CWW2003,SCW2006,TZ2002}.

Another method for using prior solutions to generate initial guesses is
polynomial extrapolation, i.e., fitting a polynomial curve to the solutions at
the previous time steps and forming the initial guess by evaluating the
polynomial at the new time step.  In the most basic version of this method, the
curve is taken to be a simple (Lagrange) interpolant, though approximations
based on Taylor expansions \cite{CWW2003} have also been proposed.  Using the
solution at the previous time step is itself a technique of this
type---zeroth-order or constant interpolation---and there are many examples in
the literature of codes and algorithms that employ first-order (linear) and
second-order (quadratic) interpolation to compute initial guesses based on the
previous two or three time steps, respectively.  Extrapolation based on
higher-order polynomial interpolation has also been considered
\cite{BILO1997,GK2011,PH2018,Sht2015}, but the results are mixed in general,
with authors typically pinning any inefficacy on the poor behavior of
polynomial interpolation in \equispaced{} points.  Because of this, it would
seem that polynomial extrapolation is inherently limited compared to
projection: while the latter can potentially profitably use data from as many
previous time steps as one is willing to store, the former is limited to using
just a handful of time steps.  Moreover, right-hand side projection is
\emph{guaranteed} to provide a reduction in the residual norm with each
additional stored datum.  As Fischer points out in \cite{Fis1998}, there is no
such guarantee for extrapolation.

This article has two primary purposes.  The first is to revisit these initial
guess techniques---projection and extrapolation---in the context of PDE solvers
running on modern high-performance computing (HPC) architectures that employ
graphics processing units (GPUs) as accelerators.  GPUs can be challenging to
program effectively; naive implementations of algorithms on GPUs are rarely the
most efficient.  Nevertheless, as GPU-based architectures become increasingly
common%
\footnote{At the time of this writing, the two most powerful HPC systems in the
United States---\emph{Summit} at Oak Ridge National Laboratory and
\emph{Sierra} at Lawrence Livermore National Laboratory---derive the majority
of their computing power from GPUs.  Moreover, the \emph{Aurora},
\emph{Frontier}, and \emph{El Capitan} supercomputers to be delivered in
2021-2022 as part of the U.S. Department of Energy's Exascale Computing Project
will all have GPU-based architectures.}
the demand for implementations of algorithms that use the GPU hardware
efficiently is rising.

The second purpose is to show that the outcome of the competition between
projection and extrapolation for producing superior initial guesses is not as
clear cut as the arguments made above might suggest.  First, we propose a
\emph{stabilized polynomial extrapolation} method based on solving a
least-squares problem instead of naive Lagrange interpolation
\cite{DT2018,Rak2007}, an idea that seems to be new within the context of flow
solvers, though it has appeared before in the computational chemistry
literature \cite{HH2005,PF2004}.  We further propose a \emph{sparse polynomial
extrapolation} method, which computes a Lagrange interpolant through a subset
of the retained history data corresponding to time points with a more favorable
distribution for polynomial interpolation, analogous to the ``mock-Chebyshev''
subsampling idea proposed in \cite{BX2009}.  We will see that while projection
does tend to provide greater reductions in Krylov solver iteration counts
compared with either of these extrapolation methods, it does so at a
comparatively high cost:  the process of extrapolation is considerably cheaper,
amounting to taking a single, fixed linear combination of the history data,
whereas projection additionally requires orthogonalization among other things.
Extrapolation moves considerably less data through the GPU (or CPU), performs
fewer floating-point operations, and, unlike projection, does not require any
global reductions (e.g., inner products), so it may be carried out in parallel
without the need for inter-node communication.  Extrapolation also requires
approximately half as much storage as projection and is simple to implement,
even on a GPU.  Our experiments show that these advantages make extrapolation
generally competitive with projection, if not slightly superior to it in some
circumstances.

We proceed as follows.  In Section \ref{SEC:FSRReview}, we review the method of
Fischer and the improvement thereto due to Christensen.  In Section
\ref{SEC:Extrapolation}, we introduce our new initial guess schemes based on
stabilized and sparse polynomial extrapolation.  As our principle application
of interest is computational fluid dynamics, we then in Section
\ref{SEC:INSSolver} describe a GPU-accelerated solver for the incompressible
Navier--Stokes (INS) equations that we use as a test-bed.  In Section
\ref{SEC:Numerics}, we discuss our GPU implementation of these initial guess
methods and illustrate their performance within our INS solver.  Additionally,
we make observations about how the effectiveness of these methods changes with
the Reynolds number of the flow being simulated.  Both our INS solver and our
implementation of the initial guess schemes are distributed with the
\texttt{libParanumal} collection of GPU-accelerated flow solvers
\cite{ChalmersKarakusAustinSwirydowiczWarburton2020}.

\section{Fischer's Successive Right-Hand-Side Projection Method}
\label{SEC:FSRReview}

In purely linear algebraic terms, Fischer's method for generating initial
guesses via successive right-hand-side projection works as follows.  Let $A$ be
an invertible square matrix, and suppose that we have found the solutions $x_1,
\ldots, x_n$ to a sequence of linear systems governed by $A$ with
corresponding right-hand sides $b_1, \ldots, b_n$ and that we want to compute
the solution $x_{n + 1}$ to the system with right-hand side $b_{n + 1}$.  That
is, we have $Ax_1 = b_1, Ax_2 = b_2, \ldots Ax_n = b_n$, and we want to solve
$Ax_{n + 1} = b_{n + 1}$.

The key idea is to consider the orthogonal projection $\hat{b}_{n + 1}$ of
$b_{n + 1}$ onto the span of $b_1, \ldots, b_n$.  Writing $\hat{b}_{n + 1}$ as
a linear combination of $b_1, \ldots b_n$,
\[
\hat{b}_{n + 1} = \alpha_1 b_1 + \cdots + \alpha_n b_n,
\]
we have $A\hat{x}_{n + 1} = \hat{b}_{n + 1}$, where $\hat{x}_{n + 1}$ is the
same linear combination of $x_1, \ldots, x_n$:
\[
\hat{x}_{n + 1} = \alpha_1 x_1 + \cdots + \alpha_n x_n.
\]
This choice of $\hat{x}_{n + 1}$ is optimal in the sense that it minimizes the
residual norm $\| b_{n + 1} - Ax \|_2$ among all linear combinations $x$ of
$x_1, \ldots, x_n$.

If $\hat{b}_{n + 1} \approx b_{n + 1}$---that is, if $b_{n + 1}$ almost lies in
the span of $b_1, \ldots, b_n$---then we will have $x_{n + 1} \approx
\hat{x}_{n + 1}$, provided that $A$ is not badly conditioned.  Thus,
$\hat{x}_{n + 1}$ ought to make a good initial guess for a Krylov subspace
method applied to $Ax_{n + 1} = b_{n + 1}$.  In particular, this is likely to
happen when $b_1, \ldots, b_{n + 1}$ are the right-hand sides for linear
systems at a group of successive time steps in a PDE simulation, as they are
frequently constructed using similar formulae applied to similar inputs and
thus tend to resemble one another fairly closely.

This procedure is simple to implement.  All one needs is a means for computing
the orthogonal projection $\hat{b}_{n + 1}$, and this can be accomplished
straightforwardly using the Gram--Schmidt process to compute an orthonormal
basis for $\linspan\{b_1, \ldots, b_n\}$.  Expanding the basis to retain
information from the linear systems solved at each new time step is easy:  just
do one additional step of Gram--Schmidt to orthogonalize the new right-hand
side against the existing basis.  Some form of reorthogonalization is necessary
to guard against loss of orthogonality in finite-precision arithmetic.  The
standard strategy is to simply do each orthogonalization twice, which is
(provably) usually sufficient \cite[\S6.9]{Par1998}.

There are two issues that must be addressed.  First, when expanding the history
space of retained right-hand sides, one must guard against adding vectors to
the basis that are too similar to the ones already there (in the sense of near
linear dependence).  We have observed that failure to do this diminishes the
quality of later computed initial guesses and can even result in the simulation
becoming unstable.  Fortunately, this is easy enough to detect:  if the norm of
the residual after projection, $\|b_{n + 1} - \hat{b}_{n + 1}\|$, is very small
relative to $b_{n + 1}$, then $b_{n + 1}$ already nearly belongs to
$\linspan\{b_1, \ldots, b_n\}$, so adding it to the history space is pointless.
The threshold for making this determination will depend on the parameters of
the PDE solver under consideration; at a minimum, it will depend on the
stopping criterion for the linear solver.

Second, in the absence of unlimited storage, it is not possible to continue
expanding the space of retained right-hand sides indefinitely:  one needs
storage for two vectors---one for the right-hand side and one for the
corresponding solution---for each time step whose history is to be kept in the
space.  In \cite{Fis1998}, Fischer addressed this problem by periodically
restarting the technique---discarding the space built to that point and
beginning anew---after taking a given fixed number of time steps.  We refer to
the resulting initial guess generation algorithm as the ``classic'' variant of
Fischer's technique, and pseudocode for it is presented in Algorithm
\ref{ALG:FSRClassic}.

\begin{algorithm2e}[t]
\SetNoFillComment \DontPrintSemicolon \setcounter{AlgoLine}{0}
\caption{``Classic'' right-hand side projection technique}\label{ALG:FSRClassic}
\SetKwInOut{Input}{Input}
\SetKwInOut{Output}{Output}
\Input{$b$, $\tilde{B}$, $\tilde{X}$, $d$, $M$, $x_0$ as described in the text}
\Output{Solution to $Ax = b$, updated history spaces and dimensions}
\lIf{$\mathrm{d} > 0$} {$x_0 = \tilde{X}\tilde{B}_{:, 1:d}^Tb$ }\label{ALL:fsrBasisInnerProducts1}
Solve $Ax = b$ with initial guess $x_0$ to produce approximate solution $x$.\;
$\tilde{x} \gets x$, $\tilde{b} \gets A\tilde{x}$\label{ALL:ApplyOp}\;
\If{$d = 0$ or $d \geq M$}{
  $\tilde{X}_{:, 1} \gets \tilde{x}/\|\tilde{b}\|$, $\tilde{B}_{:, 1} \gets \tilde{b}/\|\tilde{b}\|$\;
  $d \gets 1$\;
} \Else {
  \For(\tcp*[h]{Twice-iterated Gram-Schmidt}){$n \gets 1$ \KwTo $2$}{
    $c \gets \tilde{B}_{:, 1:d}^T\tilde{b}$\label{ALL:fsrBasisInnerProducts2}\;
    $\tilde{b} \gets \tilde{b} - \tilde{B}_{:, 1:d}c$, $\tilde{x} \gets \tilde{x} - \tilde{X}_{:, 1:d}c$\label{ALL:fsrReconstruct}\;
  }
  \If{$\|\tilde{b}\| > \varepsilon$}{
    $\tilde{X}_{:, d + 1} \gets \tilde{x}/\|\tilde{b}\|$, $\tilde{B}_{:, d + 1} \gets \tilde{b}/\|\tilde{b}\|$\label{ALL:fsrUpdate}\;
    $d \gets d + 1$\;
  }
}
\end{algorithm2e}

This wholesale discarding of the history space seems wasteful, and it is
natural to wonder if there might be a better way.  In \cite{Chr2017},
Christensen proposed dropping a single time step's worth of information from
the history space at a time using a standard linear algebra technique for
updating Gram--Schmidt QR factorizations \cite{DGKS1976},
\cite[\S6.5.2]{GV2013_4e}.  Once the history space reaches a preselected
maximum dimension, Christensen applies the update technique to eliminate the
right-hand side corresponding to the oldest retained time step,.  Once the
linear system is solved, the new right-hand side from the current time step can
then be added to the history space in the usual way.  Thus, the history space
forms a ``rolling window'' of information from the previous few time steps, and
for this reason, we refer to this method as the ``rolling QR'' variant of this
algorithm.  Pseudocode may be found in Algorithm \ref{ALG:FSRQR}.

\begin{algorithm2e}
\SetNoFillComment \DontPrintSemicolon \setcounter{AlgoLine}{0}
\caption{``Rolling QR'' right-hand side projection technique}\label{ALG:FSRQR}
\SetKwInOut{Input}{Input}
\SetKwInOut{Output}{Output}
\Input{$b$, $\tilde{B}$, $\tilde{X}$, $d$, $M$, $x_0$, $R$ as described in the text}
\Output{Solution to $Ax = b$, updated history spaces and dimensions}
\lIf{$\mathrm{d} > 0$} {$x_0 = \tilde{X}\tilde{B}_{:, 1:d}^Tb$ }
Solve $Ax = b$ with initial guess $x_0$ to produce approximate solution $x$.\;
$\tilde{x} \gets x$, $\tilde{b} \gets A\tilde{x}$\;
\If{$d = M$}{
\tcp{QR update}
$R \gets R_{:, 2:d}$\label{ALL:fsrDropQRFirstColumn1}\;
\For{$i \gets 1$ \KwTo $d - 1$} {
  $a \gets R_{i, i}$, $b \gets R_{i, i + 1}$\;
  $c \gets |a|/\sqrt{a^2 + b^2}$\;
  $s \gets \sign(a)/\sqrt{a^2 + b^2}$\;
  $G \gets \begin{bmatrix}c & s \\ -s & c\end{bmatrix}$ \tcp*{Givens rotation to zero $R_{i, i + 1}$}
  $R_{i:i + 1, :} \gets GR_{i:i + 1, :}$ \tcp*{Update $R$}
  $\tilde{B}_{:, i:(i + 1)} \gets \tilde{B}_{:, i:(i + 1)}G^T$ \tcp*{Update right-hand side history}
  $\tilde{X}_{:, i:(i + 1)} \gets \tilde{X}_{:, i:(i + 1)}G^T$ \label{ALL:fsrDropQRFirstColumn2}\tcp*{Update solution history}
}
$d \gets d - 1$\;
}
\If{$d = 0$}{
  $R_{1, 1} \gets \|\tilde{b}\|$\;
  $\tilde{X}_1 \gets \tilde{x}/R_{1, 1}$, $\tilde{B}_1 \gets \tilde{b}/R_{1, 1}$\;
  $d \gets 1$\;
} \Else {
  $R_{1:d, d + 1} \gets 0$\;
  \For(\tcp*[h]{Twice-iterated Gram-Schmidt}){$n \gets 1$ \KwTo $2$}{
    $c \gets \tilde{B}_{:, 1:d}^T\tilde{b}$\;
    $\tilde{b} \gets \tilde{b} - \tilde{B}_{:, 1:d}c$, $\tilde{x} \gets \tilde{x} - \tilde{X}_{:, 1:d}c$\;
    $R_{1:d, d + 1} \gets R_{1:d, d + 1} + c$\;
  }
  \If{$\|\tilde{b}\| > \varepsilon$}{
    $R_{d + 1, d + 1} \gets \|\tilde{b}\|$\;
    $\tilde{X}_{:, d + 1} \gets \tilde{x}/R_{d + 1, d + 1}$, $\tilde{B}_{:, d + 1} \gets \tilde{b}/R_{d + 1, d + 1}$\;
    $d \gets d + 1$\;
  }
}
\end{algorithm2e}

The listings in Algorithms \ref{ALG:FSRClassic} and \ref{ALG:FSRQR} employ the
following notation.  The vector $b$ is the new right-hand side for which a
solution is sought.  The columns of the matrices $\tilde{B}$ and $\tilde{X}$
contain the history space of previous right-hand sides and corresponding
solutions, respectively, so that $A\tilde{X} = \tilde{B}$.  The columns of
$\tilde{B}$ are orthonormal.  The variable $d$ is the current dimension of the
history spaces, and $M$ is the maximum allowable dimension.  The number
$\varepsilon$ is a tolerance for determining whether $b$ is too nearly linearly
dependent on the columns of $\tilde{B}$ to warrant adding it to the space once
the system for it is solved.  The vector $x_0$ is an arbitrary input for the
initial guess when $d = 0$.  In Algorithm \ref{ALG:FSRQR}, $R$ is the
triangular factor in the $QR$ decomposition of the right-hand side history
space (in which $\tilde{B}$ plays the role of $Q$).  Matrices are indexed using
a MATLAB-like colon notation, e.g., $\tilde{B}_{:, i:(i + 1)}$ refers to
columns $i$ and $i + 1$ of $\tilde{B}$.  For any matrix $L$, the symbol $L^T$
denotes the transpose of $L$.

\section{Stabilized Polynomial Extrapolation Methods}
\label{SEC:Extrapolation}

Compared with projection, extrapolation is both an older and more widely-used
method for computing initial guesses.  Like projection, it forms the guess as a
linear combination of prior solutions, but it chooses the coefficients of the
combination differently.

In our discussion, we employ the following notation, which differs slightly
from that of the previous section:  given solutions $x_{n - M + 1}, \ldots,
x_n$ to systems at $M$ previous time steps, we form an estimate $\hat{x}_{n +
1}$ for the solution $x_{n + 1}$ at the current time step.  We denote the
coefficients of the linear combination that forms $\hat{x}_{n + 1}$ by
$\beta_i$, writing
\begin{equation}\label{EQN:ExtrapExpn}
\hat{x}_{n + 1} = \sum_{i = 1}^M \beta_i x_{n - M + i},
\end{equation}
i.e., $\beta_1$ multiplies the oldest stored value, $\beta_2$ the next oldest,
and so forth.  To keep the discussion concrete, we focus primarily on the case
of a fixed time step, but our remarks apply to variable time step schemes as
well with only minor modifications.

\subsection{Extrapolation via Lagrange Interpolation}
\label{SSEC:LagrangeExtrap}

The most common extrapolation technique by far is simple Lagrange polynomial
interpolation:  compute a polynomial of degree $M - 1$ that interpolates $x_{n
- M + 1}, \ldots, x_n$ and determine $\hat{x}_{n + 1}$ by evaluating that
polynomial at the current time step.  Using the solution at the previous time
step is one example of this strategy with a degree-0 (constant) interpolating
polynomial; this corresponds to the case where $M = 1$ and $\beta_1 = 1$.  As
is widely known, if the size of the time step is constant, the coefficients
$\beta_i$ have a simple closed-form expression:

\begin{theorem}
For extrapolation of data given at $M$ \equispaced{} points via Lagrange
interpolation by a polynomial of degree $M - 1$,
\begin{equation}\label{EQN:NaiveExtrapCoeffs}
\beta_i = (-1)^{M - i} \binom{M}{i - 1}.
\end{equation}
\end{theorem}

The proof, which we omit, is a straightforward application of Newton's forward
difference formula \cite[25.2.28]{AS1968}, \cite{Joy1971} and standard
identities involving binomial coefficients.  As an example, extrapolation via a
linear fit with a fixed time step has $M = 2$, $\beta_1 = -1$, and $\beta_2 =
2$.  These coefficients are tabulated in several places in the literature, e.g.
\cite[Table 1]{Sht2015}, and, ignoring sign, can be found on the $(M +
1)^\textrm{th}$ row of Pascal's triangle.

This scheme is simple but runs into trouble in practice as $M$ increases.  This
is typically blamed on the fact that the polynomial interpolation problem can
be poorly conditioned if the interpolation points are not chosen appropriately.
Equi-spaced interpolation in particular (which corresponds in our context to
Lagrange interpolation of solution history with a fixed time step) is known to
become exponentially poorly conditioned as the number of points increases and
moreover can fail to converge even in theory for well-behaved functions, a
property known as the Runge phenomenon \cite{Tre2013}.  The resulting
algorithmic instability limits the utility of extrapolation via Lagrange
interpolation to small values of $M$ and low polynomial orders.  This would
seem to be a severe disadvantage for extrapolation compared to projection
methods, which can use as much history data as one is willing to store to
further improve the initial guess.

\subsection{Stable Polynomial Extrapolation via Least-Squares}

We posit that this limitation can be overcome by switching from interpolation
to least-squares fitting.  While polynomial interpolation at \equispaced{}
points deserves its poor reputation, recent work \cite{DT2018,Rak2007} has
shown that polynomial least-squares fitting can be performed stably provided
that the number of data points $M$ increases at a rate proportional to the
\emph{square} of the degree $m$ of the polynomial being fit:  $M = O(m^2)$ as
$m$ becomes large.  While this technique has been known to approximation
theorists for some time, it seems to have been largely overlooked for use in
generating initial guesses within PDE solvers, despite some precedence in
\cite{HH2005,PF2004}.

The price paid for this stability is a reduction in the order of accuracy:
even when the data to be fit consists of samples of values of a holomorphic
function, the rate of convergence of the scheme is only root-exponential in $M$
as $M$ increases, while a well-constructed interpolation scheme using good
points would converge geometrically.  This trade-off is inevitable, as it is
impossible to construct a stable algorithm for approximating data from
\equispaced{} samples that converges at a geometric rate
\cite{platte2011impossibility}.  Nevertheless, as our experiments in Section
\ref{SEC:Numerics} show, the initial guesses generated by this \emph{stabilized
polynomial extrapolation} scheme can be quite effective at reducing solver
iteration counts and the time to solution.

Given a polynomial degree $m$ and a history space of size $M \geq m + 1$, we
can compute the coefficients $\beta_i$ as follows.  First, we lose no
generality in assuming that $x_{n - M + 1}, \ldots x_n$ are scalar values
rather than vectors, as vector data is handled merely by extrapolating each
component.  Since we assume a fixed time step, we also lose no generality in
assuming that the time points at which the history data are provided are the
equally spaced points $t_i = -1 + (i - 1)h$ in $[-1, 1]$, where $h = 2/(M - 1)$
and $1 \leq i \leq M$.  For if not, we can make them so by applying an affine
transformation, which does not change the solution to the least-squares
problem.  Under this assumption, the new time point---the point at which the
extrapolating polynomial must be evaluated---is $t_{n + 1} = 1 + h$.

Let $\psi_0, \ldots \psi_m$ be any basis for the space of polynomials of degree
at most $m$ that is well-behaved on $[-1, 1]$, such as Legendre or Chebyshev
polynomials, and consider the Vandermonde-like matrix $V$ whose $(i, j)$ entry
is $V_{ij} = \psi_j(t_i)$ for $1 \leq i \leq M$ and $0 \leq j \leq m$.
Expressing the extrapolation polynomial in this basis as $\psi(t) = c_0
\psi_0(t) + \cdots + c_m\psi_m(t)$, the overdetermined linear system for the
expansion coefficients $c_0, \ldots c_m$ is
\begin{equation}\label{EQN:LSQRSystem}
Vc = x,
\end{equation}
where
\[
c = \begin{bmatrix} c_0 \\ \vdots \\ c_m \end{bmatrix} \qquad\text{and}\qquad x = \begin{bmatrix} x_{n - M + 1} \\ \vdots \\ x_n \end{bmatrix}.
\]
Noting that $V$ is full-rank, the unique least-squares solution to this system
is
\[
c = (V^TV)^{-1}V^Tx.
\]
With $c$ in hand, we evaluate
\[
\hat{x}_{n + 1} = \psi(t_{n + 1}) = v^T c,
\]
where $v$ is a column vector of length $m + 1$ with $j^\textrm{th}$ component
$v_j = \psi_j(t_{n + 1})$, $0 \leq j \leq m$.  Thus,
\[
\hat{x}_{n + 1} = v^T (V^TV)^{-1}V^Tx,
\]
from which it follows that
\begin{equation}\label{EQN:LSQRCoeffs}
\beta^T = v^T(V^TV)^{-1}V^T,
\end{equation}
where $\beta^T = \begin{bmatrix} \beta_1 & \cdots & \beta_M\end{bmatrix}$.

Note that once we have $\beta$ in hand, we no longer need $c$; the expansion
\eqref{EQN:ExtrapExpn} does not use the least-squares polynomial directly.  In
particular, we do not need to carry out the fitting procedure just described
for each component of the vectors.

\subsection{Sparse Polynomial Extrapolation}
\label{SSEC:SparseExtrap}

There is another way to characterize the coefficients $\beta_1, \ldots,
\beta_M$ just computed that turns out to be insightful.  Observe that
\eqref{EQN:LSQRCoeffs} implies
\begin{equation}\label{EQN:PolyExactness}
\beta^T V = v^T.
\end{equation}
To interpret this equation, recall that the $j^\textrm{th}$ column of $V$
consists of the values of $\psi_j$ at the prior time points, so the inner
product of $\beta$ with this column gives the value obtained by the
extrapolation scheme when the history data consist of samples of $\psi_j$.
Since the $j^\textrm{th}$ element of $v$ is the value of $\psi_j$ at the new
time point, it follows that the extrapolation scheme is \emph{exact} in this
case.  It then follows by linearity that the scheme is exact if the history
data consist of samples of any polynomial of degree at most $m$.  Of course we
do not need \eqref{EQN:PolyExactness} to tell us this; it follows from the fact
that $c$ was chosen to solve \eqref{EQN:LSQRSystem} in a least-squares sense.
If $x$ belongs to $\Ran(V)$, then $c$ can be chosen so that the residual
vanishes.

But what if, instead of \eqref{EQN:LSQRSystem}, we take
\eqref{EQN:PolyExactness} as our starting point and demand \emph{only} that our
extrapolation scheme yield exact results for degree-$m$ polynomial data?  The
equation \eqref{EQN:PolyExactness} defines an underdetermined system for the
unknown row vector $\beta^T$ and thus has infinitely many solutions.  The
choice of $\beta^T$ given by \eqref{EQN:LSQRCoeffs} has the distinction of
being the solution to \eqref{EQN:PolyExactness} that has minimum norm, a
consequence of the fact that the matrix $(V^TV)^{-1}V^T$ is the
(Moore--Penrose) pseudoinverse of $V$.  In essence, this minimality property
says that the extrapolation scheme derived from \eqref{EQN:LSQRCoeffs} is the
``most stable'' of all that can be derived from \eqref{EQN:PolyExactness}, and
in general, this is a good reason to prefer it.  Nevertheless, it is natural to
ask if there are other solutions to \eqref{EQN:PolyExactness} that yield
extrapolation schemes of interest.

As a potential alternative to the pseudoinverse solution, we propose to use
extrapolation schemes derived from solutions to \eqref{EQN:PolyExactness}
obtained via a column-pivoted QR factorization \cite[\S5.6.2]{GV2013_4e}.
First, observe that \eqref{EQN:PolyExactness} is equivalent to
\[
V^T \beta = v.
\]
Write $V^TP = QR$, where $Q$ is an $(m + 1) \times (m + 1)$ orthogonal matrix,
$R$ is an $(m + 1) \times M$ upper trapezoidal matrix, and $P$ is an $M \times
M$ permutation matrix chosen by the column pivoting.  Then, since $P^{-1} =
P^T$, we have
\[
QR P^T \beta = v.
\]
Now, partition $R$ and $P^T\beta$ as
\[
R = \begin{bmatrix} \hat{R} & \tilde{R} \end{bmatrix} \qquad\text{and}\qquad P^T \beta = \begin{bmatrix} \hat{\beta} \\ \tilde{\beta}\end{bmatrix},
\]
where $\hat{R}$ is $(m + 1) \times (m + 1)$, $\tilde{R}$ is $(m + 1) \times (M
- m - 1)$ and $\hat{\beta}$ and $\tilde{\beta}$ are length $m + 1$ and $M$,
respectively.  Then, we have
\[
\hat{R}\hat{\beta} + \tilde{R}\tilde{\beta} = Q^Tv.
\]
and we are free to choose $\hat{\beta}$ and $\tilde{\beta}$ to make this
equation hold.  One natural choice is to set $\tilde{\beta} = 0$.  Then, since
$V$ is full-rank, $\hat{R}$ is invertible, and we have $\hat{\beta} =
\hat{R}^{-1}Q^Tv$. It follows that
\begin{equation}\label{EQN:CPQRCoeffs}
\beta = P\begin{bmatrix}\hat{R}^{-1}Q^Tv \\ 0 \end{bmatrix}
\end{equation}
is another choice for $\beta$ that satisfies \eqref{EQN:PolyExactness}.

The advantage of choosing $\beta$ according to \eqref{EQN:CPQRCoeffs} instead
of \eqref{EQN:LSQRCoeffs} is that with the former, only $m + 1$ of the entries
are nonzero, which reduces the amount of computational effort required to
evaluate \eqref{EQN:ExtrapExpn}.  For this reason, we refer to this scheme as
\emph{sparse polynomial extrapolation}.  The sacrifice made is a little bit of
stability:  in general, \eqref{EQN:CPQRCoeffs} is not the minimum-norm solution
to \eqref{EQN:PolyExactness}, so extrapolation with \eqref{EQN:CPQRCoeffs} will
have a greater propensity to ``amplify'' the history data.  Nevertheless, the
column pivoting ensures that this loss of stability is minimal, and we will see
in Section \ref{SSEC:SolverItersSimTime} that both \eqref{EQN:LSQRSystem} and
\eqref{EQN:CPQRCoeffs} lead to successful extrapolation schemes in practice.

Another way to view this scheme is as follows.  Since only $m + 1$ of the
coefficients are nonzero, yet \eqref{EQN:CPQRCoeffs} solves the exactness
equation \eqref{EQN:PolyExactness} for polynomials, it follows that an
extrapolation scheme based on \eqref{EQN:CPQRCoeffs} is actually computing a
degree-$m$ polynomial interpolant through the data points corresponding to the
nonzero coefficients.  What sets this interpolation scheme apart from the basic
one discussed in Section \ref{SSEC:LagrangeExtrap} is which points are chosen.
The basic scheme simply uses the last $m + 1$ data points, while the scheme
based on \eqref{EQN:CPQRCoeffs} adaptively selects $m + 1$ of the previous $M$
data points to yield a scheme with better stability properties.

This idea is not without precedent in the literature.  Bos and coauthors used
column-pivoted QR in \cite{bos2010computing} to select good sets of points for
multivariate interpolation.  In \cite{BX2009}, Boyd and Xu proposed what they
call ``mock-Chebyshev subset interpolation'', in which they select $m + 1$
interpolation points from an \equispaced{} grid of $O(m^2)$ points so that they
are distributed similarly to Chebyshev points, which are a near-optimal set of
points for polynomial interpolation \cite{Tre2013}.  The points selected by the
column-pivoted QR procedure exhibit similar end point clustering to Chebyshev
points.

\section{Comparison of Computational Costs}

It would seem at first glance that the minimality property satisfied by the
projection schemes makes them obviously superior.  In more detail, we have the
following result:

\begin{theorem}
Consider a sequence $Ax_{n - M + 1} = b_{n - M + 1}, \ldots, Ax_n = b_n$ of $M$
linear systems, where $x_i = x(ih)$ and $b_i = b(ih)$ are equally-spaced
samples of smooth functions $x(t)$ and $b(t)$ for some $h > 0$.  Let
$\hat{x}_{n + 1}^P$ denote the approximation to $x_{n + 1} = A^{-1}b_{n + 1}$
obtained using right-hand side projection onto $\linspan \{b_{n - M + 1},
\ldots, b_M\}$, and let $\hat{x}_{n + 1}^E$ denote the approximation to $x_{n +
1}$ obtained by degree-$m$ least-squares extrapolation of $\linspan \{x_{n - M
+ 1}, \ldots, x_M\}$, where $M \geq m + 1$.  The residual norms satisfy:
\begin{enumerate}
\renewcommand{\labelenumi}{(\roman{enumi})}
\item $\|b_{n + 1} - A\hat{x}_{n + 1}^P\|_2 \leq \|b_{n + 1} - A\hat{x}_{n + 1}^E\|_2$.
\item $\|b_{n + 1} - A\hat{x}_{n + 1}^P\|_2 = O(h^M)$ as $h \to 0$.
\item $\|b_{n + 1} - A\hat{x}_{n + 1}^E\|_2 = O(h^{m + 1})$ as $h \to 0$.
\end{enumerate}
\end{theorem}

\begin{proof}

Statement (i) follows from the fact that $\hat{x}_{n + 1}^P$ is chosen to
minimize $\|b_{n + 1} - Ax\|_2$ over all possible linear combinations $x$ of
$x_{n - M + 1}, \ldots x_n$, and $\hat{x}_{n + 1}^E$ is one such combination.
Statement (ii), proved in \cite[Theorem 2.3]{AS2012}, is a consequence of the
fact that another such combination is given by degree-$(M - 1)$ polynomial
interpolation, and this can be shown to be $O(h^M)$-accurate using Taylor's
theorem.

Statement (iii) can be proved by similar arguments.  If $\hat{x}_{n + 1}^I$ is
the initial guess obtained by degree-$m$ polynomial interpolation in $x_{n -
m}, \ldots, x_n$, then
\[
\|x_{n + 1} - \hat{x}_{n + 1}^E\|_2 \leq \|x_{n + 1} - \hat{x}_{n + 1}^I\|_2,
\]
since $\hat{x}_{n + 1}^I$ is a competitor in the least-squares problem
$\hat{x}_{n + 1}^E$ was chosen to solve.  Thus,
\[
\|b_{n + 1} - A\hat{x}_{n + 1}^E\|_2 \leq \|A\|_2 \|x_{n + 1} - \hat{x}_{n + 1}^E\|_2 \leq \|A\|_2 \|x_{n + 1} - \hat{x}_{n + 1}^I\|_2 = O(h^{m + 1})
\]
by the result on the accuracy of polynomial interpolation just mentioned.
\end{proof}

We will see in Section \ref{SEC:Numerics} that the advantage of projection
predicted by this theorem is real:  initial guesses derived from projection do
tend to yield smaller residuals and correspondingly lower solver iteration
counts.  But this advantage comes at a price.  Projection is more expensive to
carry out than extrapolation, and if the difference in cost is large enough,
the latter may be able to compete when effectiveness is measured in time to
solution.

The core operations involved in both projection and extrapolation are easily
seen to be memory-bound, and we can therefore compare the cost of the two by
counting the number of memory operations that our implementations of these
methods perform at each time step.  Table \ref{TAB:DataMoveCounts} shows these
counts to leading order for each of the methods we have been considering.  In
the first column, we introduce some nomenclature for the methods that we will
use throughout the remainder of the article.  ``LAST'' refers to the initial
guess strategy of using the solution at the previous time step.
``CLASSIC($M$)'' and ``QR($M$)'' refer to the classic and rolling-QR variants
of the projection methods with a size-$M$ history space.  ``EXTRAP($m$,~$M$)''
refers to least-squares polynomial extrapolation with degree $m$ over a history
space of size $M$, and ``SPEXTRAP($m$,~$M$)'' refers to the sparse variant.
The symbol $\Ntotal$ is the dimension of the linear systems.  By ``leading
order'', we mean that we have assumed $\Ntotal \gg M$. and therefore can
neglect terms in the counts that do not contain $\Ntotal$ as a factor.

\begin{table}
\begin{center}
\makebox[\textwidth][c]{
\begin{tabular}{c c c c}
\toprule
Method & Form guess & Update history space & Total \\
\midrule
LAST               & $0$               & $0$                 & $0$                 \\
CLASSIC($M$)       & $(M + 3)\Ntotal$  & $(3M + 27)\Ntotal$  & $(4M + 30)\Ntotal$  \\
QR($M$)            & $2(M + 1)\Ntotal$ & $(10M + 24)\Ntotal$ & $(12M + 26)\Ntotal$ \\
EXTRAP($m$, $M$)   & $(M + 1)\Ntotal$  & $0$                 & $(M + 1)\Ntotal$    \\
SPEXTRAP($m$, $M$) & $(m + 2)\Ntotal$  & $0$                 & $(m + 2)\Ntotal$    \\
\bottomrule
\end{tabular}
}
\end{center}
\caption{Leading-order counts of memory operations (number of floating-point
values loaded and stored) performed per time step in our implementations of
each of the initial guess methods.  Here, $M$ refers to the size of the history
space, $m$ is the polynomial degree used in extrapolation, and $\Ntotal$ is the
dimension of the linear systems.  By ``leading-order'', we mean that we assume
$\Ntotal \gg M$.  The count for CLASSIC($M$) is an average over one restart
cycle.  Counts for the other methods assume a full history space.}
\label{TAB:DataMoveCounts}
\end{table}

Details of the computations for the entries in this table may be found in
Appendix \ref{APP:OpCountDetails}.  The point is that extrapolation moves
considerably less data than the projection methods.  Asymptotically as $M$
becomes large, EXTRAP($m$,~$M$) moves 4 times less data per time step than
CLASSIC($M$) and 12 times less data than QR($M$).  For smaller $M$, the factors
will be even larger.  Sparse extrapolation moves even less data, depending on
the choice of polynomial degree.

In addition to moving less data and being simple to implement, extrapolation
has two other advantages that make it appealing.  First, it requires
approximately half the storage of the projection methods, as it needs only a
history space of solutions while the projection methods require a history space
of right-hand sides as well.  Second, the extrapolation methods are entirely
local in the sense that they do not require any communication to carry out in
parallel.  In particular, they do not require any global inner products or
global operator evaluations.  Projection methods, on the other hand, require
both:  an operator evaluation to find the right-hand side corresponding to the
approximate solution (see line \ref{ALL:ApplyOp} of Algorithm
\ref{ALG:FSRClassic}) and inner products both during orthogonalization and when
computing the projections.

\section{A GPU-Accelerated Incompressible Navier--Stokes Solver}
\label{SEC:INSSolver}

Our goal in this article is to implement these methods in a GPU-accelerated PDE
solver in such a way that they use the GPU hardware efficiently.  We consider
specifically a solver for the incompressible Navier--Stokes (INS) equations,
\begin{equation}\label{EQN:INS}
\begin{aligned}
\frac{\partial \bdu}{\partial t} -\nu \nabla^2 \bdu + \bdu \cdot \nabla \bdu + \nabla p &= 0 \\
\nabla \cdot \bdu &= 0,
\end{aligned}
\end{equation}
posed in a closed domain $\Omega$ in $\R^3$  Here, $\bdu = (u_x, u_y, u_z)$ is
the velocity field $p$ is the pressure, and $\nu$ is the kinematic viscosity.
We assume that $\partial \Omega$ may be written as $\partial \Omega_D \cup
\partial \Omega_N$, where $\partial \Omega_D$ and $\partial \Omega_N$ are
disjoint subsets of $\partial \Omega$ on which we impose the boundary
conditions
\[
\begin{aligned}
\bdu &= \bdg && \text{on $\Omega_D$} \\
\nu \nabla \bdu \cdot \hat{\bdn} &= p\hat{\bdn} && \text{on $\Omega_N$},
\end{aligned}
\]
respectively, where $\hat{\bdn}$ denotes the outward-pointing unit normal.
Standard manipulations lead to the variational formulation
\begin{equation}\label{EQN:INSVF}
\begin{aligned}
\frac{\partial}{\partial t}\int_\Omega \bdu \cdot \bdv \: dV + \nu \int_\Omega \nabla \bdu : \nabla \bdv \: dV  + \int_\Omega \bdu \cdot \nabla \bdu \: dV - \int_\Omega p(\nabla \cdot \bdv) \: dV &= 0 \\
\int_\Omega (\nabla \cdot \bdu)q &= 0,
\end{aligned}
\end{equation}
to be satisfied by the solution $\bdu$ for all test functions $\bdv$, $q$ of a
suitable regularity, where $\bdv$ vanishes on $\partial \Omega_D$.

The solver we use is a continuous Galerkin variant of the high-order
discontinuous Galerkin INS solver described in \cite{KCSW2019} and is freely
available as part of the \texttt{libParanumal} suite of flow solvers
\cite{ChalmersKarakusAustinSwirydowiczWarburton2020}.  For completeness, we
describe the discretization and the salient characteristics of the solver here.
Note that while we have described (and will continue to describe) everything in
terms of flow problems in three spatial dimensions, a two-dimensional solver
may be constructed following a similar prescription.

\subsection{Spatial Discretization}

We partition $\Omega$ into conforming, nonoverlapping elements, which we take
in this article to be hexahedra, though other shapes are also viable.  We map
each element to the reference cube $[-1, 1]^3$, on which we define a space of
multivariate polynomials consisting of tensor products of univariate
polynomials of degree at most $N$ in each spatial variable.  We represent
polynomials in this space by their values on a tensor product grid of
Gauss--Legendre--Lobatto (GLL) nodes, and we seek a piecewise continuous
approximation to the solution of \eqref{EQN:INS} in $\Omega$ formed by linear
combinations of these polynomials on each element.

Thus, on element $e$, we expand the local velocity $\bdu^e$ and pressure $p^e$
as
\[
\bdu^e(\bdx) = \sum_{i = 1}^{N_p} \bdu^e_i \ell_i(\bdx) \qquad\text{and}\qquad p^e(\bdx) = \sum_{i = 1}^{N_p} p_i^e\ell_i(\bdx),
\]
where $N_p = (N + 1)^3$ is the dimension of the local polynomial space (i.e.,
the number of tensor product GLL nodes) and $\ell_i$ is the $i^\textrm{th}$
basis function (a tensor product of Lagrange basis functions for the GLL nodes
in one dimension).  Inserting these expansions into the variational form
\eqref{EQN:INSVF} and imposing the Galerkin condition, we arrive the local
semidiscrete system
\[
\begin{aligned}
\frac{d\bdu^e}{dt} + \nu \bdL^e\bdu^e + \bdN^e(\bdu^e) - \bdG^ep^e &= 0 \\
D^e\bdu^e &= 0,
\end{aligned}
\]
where we abuse notation and identify $\bdu^e$ and $p^e$ with vectors of the
expansion coefficients $\bdu_i^e$ and $p_i^e$, and where $\bdL^e$, $\bdN^e$,
$\bdG^e$, and $D^e$ are the discrete (vector) Laplacian, advection, gradient,
and divergence operators, respectively.

Demanding that the solution be continuous across across elements, we assemble
the element-local semidiscrete systems into a global semidiscrete system
\begin{equation}\label{EQN:INSGlobalSemidiscrete}
\begin{aligned}
\frac{d\bdu}{dt} + \nu \bdL\bdu + \bdN(\bdu) - \bdG p &= 0 \\
D\bdu &= 0,
\end{aligned}
\end{equation}
where we again abuse notation, identifying $\bdu$ and $p$ with vectors of
expansion coefficients of the approximate solution in the global piecewise
polynomial basis.  That is, they are vectors of the values of the computed
solution at the union over all elements of the mapped tensor product GLL grids
with redundant points shared by multiple elements on their corners, edges, and
faces eliminated.  The operators $\bdL$, $\bdN$, $\bdG$, and $D$ are assembled
from the local operators $\bdL^e$, $\bdN^e$, $\bdG^e$, and $D^e$ in the usual
way.

Our implementation is ``matrix-free'':  we do not explicitly form the global
matrices $\bdL$, $\bdG$, and $D$, nor do we work directly with the nonlinear
global advection operator $\bdN$.  Instead, we effect the action of these
matrices on vectors by applying the local operators to vectors of the local
solution values on each element and then enforcing continuity by applying a
gather--scatter operation \cite[\S4.5.1]{DFM2002}, for which we use the
\texttt{gslib} software library \cite{gslib,FLPS2008,Tuf1998}.  The local
operators $\bdL^e$, $\bdG^e$, and $D^e$ are applied using quadrature on the
tensor product GLL grid on element $e$.  To combat aliasing due to nonlinearity
when evaluating the local advection operator $\bdN^e$, we first interpolate the
solution to a tensor-product grid of Gauss--Legendre (\emph{not} Lobatto)
quadrature points of size $(N + 2)^3$ and evaluate the integral with Gauss
quadrature.

\subsection{Temporal Discretization}

We discretize \eqref{EQN:INSGlobalSemidiscrete} in time in an
implicit--explicit fashion, using a $3^\textrm{rd}$-order backward
differentiation formula to treat the diffusive $\nu L\bdu$ term and an explicit
$3^\textrm{rd}$-order polynomial extrapolation scheme for the nonlinear
advective term $\bdN(\bdu)$.  If $\Delta t$ is the time step size, and if $u^n$
and $p^n$ denote the approximations to $\bdu$ and $p$ at the $n^\textrm{th}$
time step, we obtain a system of the form
\begin{equation}\label{EQN:INSGlobalFullyDiscrete}
\begin{aligned}
\bigl(\gamma I + \nu(\Delta t)\bdL\bigr)\bdu^{n + 1} - (\Delta t) \bdG p^{n + 1} &= \sum_{i = 0}^2 \beta_i \bdu^{n - i} - (\Delta t) \sum_{i = 0}^2 \alpha_i \bdN\bigl(\bdu^{n - i}\bigr) \\
D\bdu^{n + 1} &= 0
\end{aligned}
\end{equation}
for $\bdu^{n + 1}$ and $p^{n + 1}$, given the solution at the previous time
steps.  As these schemes are not self-starting, we initialize them using
similar schemes of lower-order.  The coefficients for the schemes may be found
in \cite[Table 5.2]{KS2005}.

Solving the fully-coupled system \eqref{EQN:INSGlobalFullyDiscrete} for
$\bdu^{n + 1}$ and $p^{n + 1}$ simultaneously can be expensive.  Instead, we
use an algebraic splitting scheme to decouple the solve for $\bdu^{n + 1}$ from
that for $p^{n + 1}$.  First, we evaluate the part of the right-hand side that
is known explicitly:
\[
\bdf = \sum_{i = 0}^2 \beta_i \bdu^{n - i} - (\Delta t) \sum_{i = 0}^2 \alpha_i \bdN\bigl(\bdu^{n - i}\bigr).
\]
Next, we obtain an initial approximation $\hat{\bdu}$ to $\bdu^{n + 1}$ by
solving the first equation in \eqref{EQN:INSGlobalFullyDiscrete} assuming that
$p^{n + 1} = 0$:
\begin{equation}\label{EQN:VelocityEqn}
\begin{aligned}
\left(\bdL + \frac{\gamma}{\nu (\Delta t)}I\right)\hat{\bdu} &= \frac{1}{\nu(\Delta t)}\bdf. \\
\end{aligned}
\end{equation}
This approximation does not satisfy the incompressibility constraint.  We
rectify this by taking $p^{n + 1}$ to solve
\begin{equation}\label{EQN:PressureEqn}
Lp^{n + 1} = -\frac{\gamma}{\Delta t}D\hat{\bdu}
\end{equation}
(here, $L$ is the matrix that discretizes the scalar Laplacian) and then
assigning
\[
\bdu^{n + 1} = \hat{\bdu} - \frac{\Delta t}{\gamma} \bdG p^{n + 1}.
\]

The size of the time step for this scheme is limited by the
Courant--Friedrichs--Lewy (CFL) condition associated with the advective term
$\bdN(\bdu)$ (since we have treated this term explicitly).  We use a
semi-Lagrangian subcycling scheme to alleviate this restriction.  For further
details, see \cite[\S2.3]{KCSW2019}

\subsection{Linear Solvers and Preconditioning}

We solve the elliptic systems \eqref{EQN:VelocityEqn} and
\eqref{EQN:PressureEqn} at each time step using the preconditioned conjugate
gradient (PCG) method.  We precondition the velocity system
\eqref{EQN:VelocityEqn} using a (point) Jacobi approach, which is inexpensive
and works well at high Reynolds numbers (small $\nu$).

For the pressure system \eqref{EQN:PressureEqn}, we use a hybrid $p$-multigrid
/ algebraic multigrid ($p$MG/AMG) preconditioner in which we coarsen the
problem (in space) first by using the same discretization but with successively
lower polynomial degrees.  If coarsening to degree 1 yields a problem that is
still too large to solve directly, we further coarsen the problem using an
unsmoothed-aggregation algebraic multigrid scheme.  We use
Chebyshev-accelerated damped-Jacobi smoothing at each level of the multigrid
hierarchy.  One application of the preconditioner consists of a single V-cycle.
For more information, see \cite[Ch.\ 5]{Gan2015} and \cite{KCSW2019}.

\section{Numerical Results and Performance}
\label{SEC:Numerics}

\subsection{Test Problem and Hardware}

We test the effectiveness of using right-hand side projection techniques within
the solver just described by using them to simulate 3D channel flow past a
rectangular ``fence'' obstacle.  The mesh is composed of 8,528 hexahedral
elements; a 2D cross-section is displayed in Figure \ref{FIG:FenceMesh}.  The
velocity field is required to satisfy zero Dirichlet conditions on the walls of
the channel and along the fence.  At the entrance of the channel, we prescribe
a uniform rightward inflow condition.  At outflow, we use a natural boundary
condition.

\begin{figure}
\centering
\includegraphics[scale = 0.20]{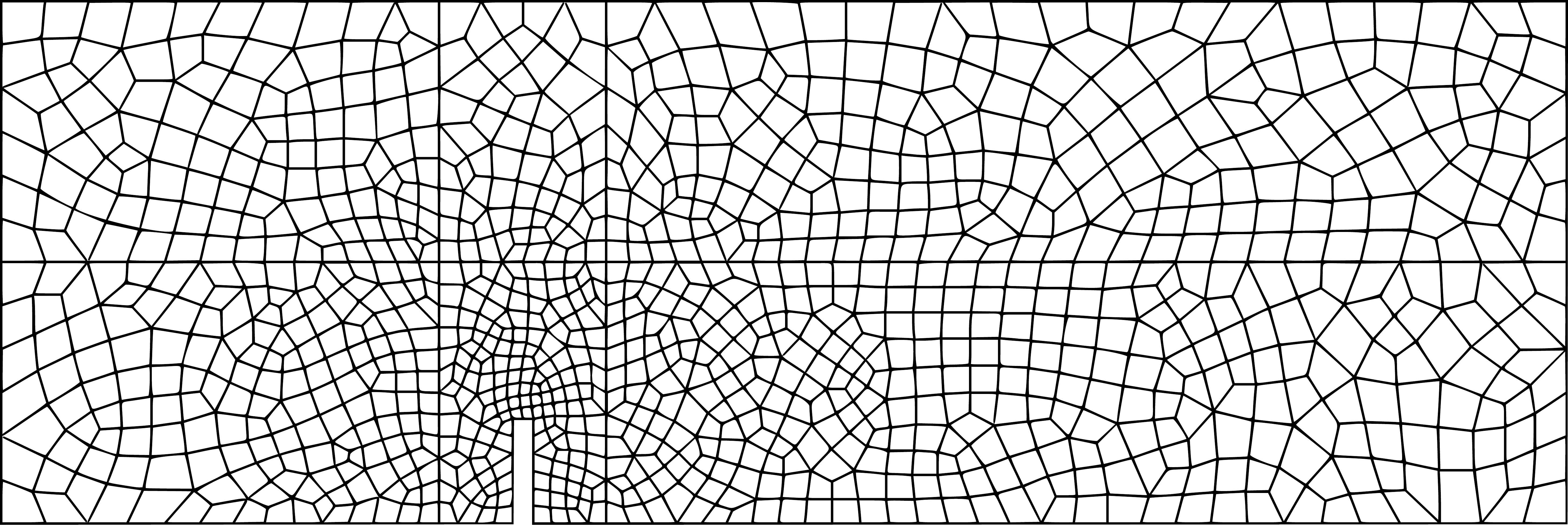}
\caption{Two-dimensional cross section of the finite element mesh for 3D
channel flow past a rectangular ``fence.''}
\label{FIG:FenceMesh}
\end{figure}

We perform all of our simulations on either a computer workstation equipped
with two NVIDIA TITAN-V GPUs or on a single node of the Cascades cluster at the
Advanced Research Computing facility at Virginia Tech.  One Cascades node is
equipped with two NVIDIA V100 GPUs.  Though we have multiple GPUs available in
both environments, for simplicity, all of our tests are confined to a single
GPU.

\subsection{GPU Efficiency}
\label{SSEC:GPUEfficiency}

We have implemented both the ``classic'' and ``rolling QR'' right-hand-side
projection techniques described in Section \ref{SEC:FSRReview} as well as the
extrapolation methods of Section \ref{SEC:Extrapolation} in the INS solver
described in the previous section.  We maintain separate history spaces of
right-hand sides for each of the three components of the velocity and for the
pressure.  The dimensions of each space may be selected independently from one
another, though we do not take advantage of this flexibility in any of our
experiments.

The projection methods require four key operations for which we have written
GPU kernels:%
\footnote{Our descriptions for the first three kernels reference lines in
Algorithm \ref{ALG:FSRClassic}.  Of course, they are used in Algorithm
\ref{ALG:FSRQR} too.}

\begin{itemize}

\item \emph{rhsProject} computes the inner products required to project the
right-hand side onto the existing basis, as is done in lines
\ref{ALL:fsrBasisInnerProducts1} and \ref{ALL:fsrBasisInnerProducts2} of
Algorithm \ref{ALG:FSRClassic}.  By blocking the inner products together, it
reduces data movement compared to the naive algorithm of computing the inner
products individually.  The inner products are computed using a pointwise
vector multiply followed by a standard binary reduction tree with an atomic add
at the end.

\item \emph{rhsReconstruct} adds or subtracts linear combinations of vectors in
the history space from a given, fixed vector.  This operation is needed in
lines \ref{ALL:fsrBasisInnerProducts1} (with the fixed vector equal to zero)
and \ref{ALL:fsrReconstruct} of Algorithm \ref{ALG:FSRClassic}.

\item \emph{rhsUpdateSpace} performs the operations on line \ref{ALL:fsrUpdate}
of Algorithm \ref{ALG:FSRClassic}, which store new vectors in the history
space.  Combining these operations, including the scalings by $\|\tilde{b}\|$,
into one kernel reduces data movement compared with performing the scalings
separately (e.g., using basic linear algebra routines) and then copying the
vectors into the spaces.

\item \emph{rhsQRUpdate} performs the QR update step in lines
\ref{ALL:fsrDropQRFirstColumn1}--\ref{ALL:fsrDropQRFirstColumn2} of Algorithm
\ref{ALG:FSRQR}.

\end{itemize}

The extrapolation methods, being considerably simpler, require just a single
kernel that computes the required linear combination of the history vectors; we
call this kernel \emph{extrapKernel}.

As mentioned above, these operations are memory-bound, and our kernels are
therefore best judged by the bandwidth they attain.  To test their performance,
we run 48 time steps of our 3D fence simulation on a single TITAN-V GPU using
polynomial degrees $N = 1, 2, \ldots 6$ and using both the classic and rolling
QR variants of the projection methods as well as the extrapolation methods with
degree $m = 2, 3, 4$.  We vary the maximum size $M$ of the history space,
trying $M = 4, 8, 12$.  On each run, we collect statistics about the kernels'
performance using NVIDIA's \texttt{nvprof} profiling utility.

These statistics are displayed in the roofline plots of Figure
\ref{FIG:RooflinePlots}, which shows the average bandwidth attained by the
kernels under each configuration as a function of the amount of data they move.
We obtained the roofline itself by measuring the bandwidth attained on a large
copy operation, which we found to be approximately 580 GB/s.%
\footnote{The datasheet for the TITAN-V indicates a theoretical peak bandwidth
of approximately 653 GB/s.}
The plots indicate that our kernels all operate close to the streaming
performance limit.

\begin{figure}[h!]
\centering
\subfloat{\begin{tikzpicture}[scale = 0.7]
  \begin{axis}[
    scale = 1.0,
    samples = 100,
    xmin = 0, xmax = 0.393090,
    title = rhsProject,
    xlabel = {Bytes read/written (GB)},
    ylabel = {Attained bandwidth (GB/s)},
  ]

  \addplot[thick, blue] file {figures/data/figRoofProject_Roof.dat};
  \addplot[mark = x, only marks, thick, red] file {figures/data/figRoofProject_Data.dat};

  \end{axis}
\end{tikzpicture}}
\subfloat{\begin{tikzpicture}[scale = 0.7]
  \begin{axis}[
    scale = 1.0,
    samples = 100,
    xmin = 0, xmax = 0.389834,
    title = rhsReconstruct,
    xlabel = {Bytes read/written (GB)},
    ylabel = {Attained bandwidth (GB/s)},
  ]

  \addplot[thick, blue] file {figures/data/figRoofReconstruct_Roof.dat};
  \addplot[mark = x, only marks, thick, red] file {figures/data/figRoofReconstruct_Data.dat};

  \end{axis}
\end{tikzpicture}} \\
\subfloat{\begin{tikzpicture}[scale = 0.7]
  \begin{axis}[
    scale = 1.0,
    samples = 100,
    xmin = 0, xmax = 0.194750,
    title = rhsUpdateSpace,
    xlabel = {Bytes read/written (GB)},
    ylabel = {Attained bandwidth (GB/s)},
  ]

  \addplot[thick, blue] file {figures/data/figRoofUpdateSpace_Roof.dat};
  \addplot[mark = x, only marks, thick, red] file {figures/data/figRoofUpdateSpace_Data.dat};

  \end{axis}
\end{tikzpicture}}
\subfloat{\begin{tikzpicture}[scale = 0.7]
  \begin{axis}[
    scale = 1.0,
    samples = 100,
    xmin = 0, xmax = 1.576220,
    title = rhsQRUpdate,
    xlabel = {Bytes read/written (GB)},
    ylabel = {Attained bandwidth (GB/s)},
  ]

  \addplot[thick, blue] file {figures/data/figRoofQRUpdate_Roof.dat};
  \addplot[mark = x, only marks, thick, red] file {figures/data/figRoofQRUpdate_Data.dat};

  \end{axis}
\end{tikzpicture}} \\
\subfloat{\begin{tikzpicture}[scale = 0.7]
  \begin{axis}[
    scale = 1.0,
    samples = 100,
    xmin = 0, xmax = 0.334442,
    title = extrapKernel,
    xlabel = {Bytes read/written (GB)},
    ylabel = {Attained bandwidth (GB/s)},
  ]

  \addplot[thick, blue] file {figures/data/figRoofExtrap_Roof.dat};
  \addplot[mark = x, only marks, thick, red] file {figures/data/figRoofExtrap_Data.dat};

  \end{axis}
\end{tikzpicture}}
\caption{Roofline plots showing the bandwidth attained by each of the four GPU
kernels described in Section \ref{SSEC:GPUEfficiency} used for implementing the
projection-based initial guess methods.  The red crosses indicate the bandwidth
attained by the kernel when the size of the input results in the given number
of memory transactions.  The blue line is an estimate of the maximum attainable
memory bandwidth, measured by timing a copy operation on a large array.}
\label{FIG:RooflinePlots}
\end{figure}
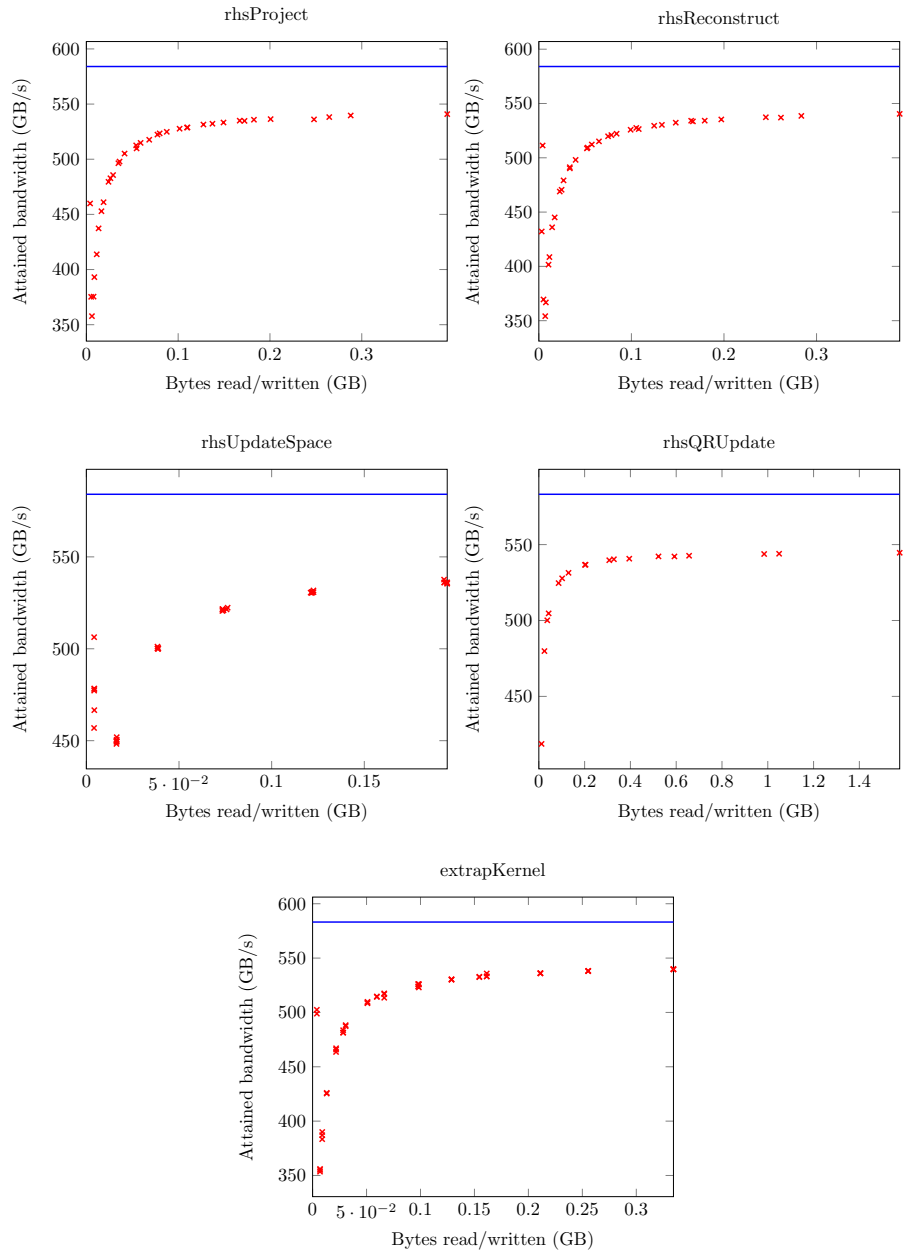

\subsection{Solver Iterations and Simulation Time}
\label{SSEC:SolverItersSimTime}

Computing initial guesses using either projection or extrapolation instead of
just using the solution at the previous time step generally results in
substantial reductions in both solver iterations (numbers of PCG steps) and
time to solution.  In fact, profiling reveals that, on average and for
``reasonable'' choices of the history space dimension, computing the projection
and forming the initial guess with the rolling QR method takes approximately
the same amount of time as one PCG iteration (including application of our
hybrid multigrid preconditioner), so it will be a net win in terms of time if
it saves just two iterations per solve.  Extrapolation is approximately ten
times faster still.

In general, these methods save considerably more than just two iterations.
Table \ref{TAB:FSRTimes} shows the time (in seconds) to run our solver on the
3D fence flow test problem described above to final time $t = 1.00$ (8737 time
steps) on a single NVIDIA V100 GPU.  We ran the solver for flows with Reynolds
numbers $\mathrm{Re} = 200, 500, 1000$ (generated by altering the value of the
viscosity $\nu$) and using nine different initial guess strategies:  the
solution at the previous time step, the classic projection method, the rolling
QR projection method, and both the least-squares and sparse extrapolation
methods with degrees $m = 2, 3, 4$.  For the projection and extrapolation
methods, we used history space dimensions $M = 4, 8, 12$.  The numbers in
parentheses indicate the speedup obtained with the projection methods relative
to using the solution at the last time step.

Even with a history space of dimension just $M = 4$ and the classic projection
scheme with its naive periodic restart procedure, we obtain a 35--40\% speedup
over using the previous time step's solution as an initial guess.  Switching to
the rolling QR scheme improves this to about 55-60\%.  While its implementation
is more complicated, the rolling QR scheme is clearly superior to the classic
one:  the classic scheme requires a history space of dimension $M = 12$ to
attain speedups that the rolling QR scheme is able to attain with $M = 4$.  We
note that $M = 8$ seems to be the dimension of diminishing returns for the
rolling QR method for this problem: there is no significant difference between
the times and speedups attained for $M = 8$ compared with $M = 12$; if
anything, $M = 12$ yields a method that is slightly slower.

The real winners here, however, are the extrapolation methods, which attain a
speedup of 70\% or greater in most cases and up to 85\% with some
configurations.  This improved time to solution comes in spite of the fact that
the projection methods, exemplified by the rolling QR scheme, attain lower
solver iteration counts, in particular for the pressure problem.  The average
PCG iteration counts per time step required for the velocity and pressure
solves are presented in Tables \ref{TAB:FSRVelocityIters} and
\ref{TAB:FSRPressureIters}, respectively.

As an example, Table \ref{TAB:FSRPressureIters}, shows that for $\mathrm{Re} =
200$, QR($8$) performs roughly 15\% fewer iterations than EXTRAP($2$,~$8$) when
solving the pressure problem, yet the latter is actually about 10\% faster
according to Table \ref{TAB:FSRTimes}.  Part of this is due to the fact that
EXTRAP($2$,~$8$) performs about 5\% fewer iterations on average than QR($8$)
for the velocity problems, but it is also due to the fact that EXTRAP($2$,~$8$)
does not have to perform the (relatively) expensive update step that QR($8$)
does.

We observe that, in general, the extrapolation methods tend to yield slightly
lower iteration counts than the projection methods for the velocity problems,
while for the pressure problems, the opposite is true.  The reason for this
difference is not immediately clear.  Note that reducing pressure iterations is
in general more valuable than reducing velocity iterations, as a pressure
iteration is considerably more expensive to effect than a velocity iteration
owing to the more sophisticated preconditioner.

\begin{table}
\centering
\begin{tabular}{c l | r r | r r | r r}
\toprule
\multicolumn{2}{c |}{Method} & \multicolumn{2}{c |}{$\mathrm{Re} = 200$} & \multicolumn{2}{c |}{$\mathrm{Re} = 500$} & \multicolumn{2}{c}{$\mathrm{Re} = 1000$} \\
\midrule
\multicolumn{2}{c |}{Last time step} & 2629 & (1.00) & 2586 & (1.00) & 2568 & (1.00) \\
\midrule\multirow{3}{*}{Classic} & $M = 4$ & 1872 & (1.40) & 1875 & (1.38) & 1901 & (1.35) \\
& $M = 8$ & 1721 & (1.53) & 1712 & (1.51) & 1748 & (1.47) \\
& $M = 12$ & 1673 & (1.57) & 1666 & (1.55) & 1701 & (1.51) \\
\midrule\multirow{3}{*}{Rolling QR} & $M = 4$ & 1620 & (1.62) & 1643 & (1.57) & 1651 & (1.56) \\
& $M = 8$ & 1589 & (1.65) & 1591 & (1.63) & 1644 & (1.56) \\
& $M = 12$ & 1625 & (1.62) & 1637 & (1.58) & 1677 & (1.53) \\
\midrule
\multirow{3}{*}{\shortstack{Extrapolation \\ (Least-sq., $m = 2$)}} & $M = 4$ & 1527 & (1.72) & 1519 & (1.70) & 1538 & (1.67) \\
& $M = 8$ & 1432 & (1.84) & 1434 & (1.80) & 1459 & (1.76) \\
& $M = 12$ & 1418 & (1.85) & 1421 & (1.82) & 1460 & (1.76) \\
\midrule
\multirow{3}{*}{\shortstack{Extrapolation \\ (Least-sq., $m = 3$)}} & $M = 4$ & 1692 & (1.55) & 1683 & (1.54) & 1677 & (1.53) \\
& $M = 8$ & 1478 & (1.78) & 1475 & (1.75) & 1475 & (1.74) \\
& $M = 12$ & 1414 & (1.86) & 1408 & (1.84) & 1411 & (1.82) \\
\midrule
\multirow{3}{*}{\shortstack{Extrapolation \\ (Least-sq., $m = 4$)}} & $M = 4$ & \multicolumn{1}{c}{---} & \multicolumn{1}{c |}{---} & \multicolumn{1}{c}{---} & \multicolumn{1}{c |}{---} & \multicolumn{1}{c}{---} & \multicolumn{1}{c}{---} \\
& $M = 8$ & 1589 & (1.65) & 1577 & (1.64) & 1570 & (1.64) \\
& $M = 12$ & 1474 & (1.78) & 1464 & (1.77) & 1468 & (1.75) \\
\midrule
\multirow{3}{*}{\shortstack{Extrapolation \\ (Sparse, $m = 2$)}} & $M = 4$ & 1531 & (1.72) & 1524 & (1.70) & 1540 & (1.67) \\
& $M = 8$ & 1420 & (1.85) & 1424 & (1.82) & 1448 & (1.77) \\
& $M = 12$ & 1405 & (1.87) & 1413 & (1.83) & 1440 & (1.78) \\
\midrule
\multirow{3}{*}{\shortstack{Extrapolation \\ (Sparse, $m = 3$)}} & $M = 4$ & 1693 & (1.55) & 1682 & (1.54) & 1677 & (1.53) \\
& $M = 8$ & 1492 & (1.76) & 1483 & (1.74) & 1483 & (1.73) \\
& $M = 12$ & 1416 & (1.86) & 1412 & (1.83) & 1418 & (1.81) \\
\midrule
\multirow{3}{*}{\shortstack{Extrapolation \\ (Sparse, $m = 4$)}} & $M = 4$ & \multicolumn{1}{c}{---} & \multicolumn{1}{c |}{---} & \multicolumn{1}{c}{---} & \multicolumn{1}{c |}{---} & \multicolumn{1}{c}{---} & \multicolumn{1}{c}{---} \\
& $M = 8$ & 1610 & (1.63) & 1595 & (1.62) & 1594 & (1.61) \\
& $M = 12$ & 1489 & (1.77) & 1476 & (1.75) & 1481 & (1.73) \\
\bottomrule
\end{tabular}
\caption{Simulation times (in seconds) for the 3D fence flow test problem to a
final time of $t = 1.00$ (approximately 8700 time steps) for each of the
initial guess strategies.  The parameter $M$ is the maximum dimension of the
history space used by the methods.  The numbers in parentheses give the
speedups attained by the methods relative to the strategy of using the solution
at the last time step.}
\label{TAB:FSRTimes}
\end{table}

\begin{table}
\centering
\begin{tabular}{c l | r r | r r | r r}
\toprule
\multicolumn{2}{c |}{Method} & \multicolumn{2}{c |}{$\mathrm{Re} = 200$} & \multicolumn{2}{c |}{$\mathrm{Re} = 500$} & \multicolumn{2}{c}{$\mathrm{Re} = 1000$} \\
\midrule
\multicolumn{2}{c |}{Last time step} & 26.00 & (1.00) & 19.68 & (1.00) & 15.36 & (1.00) \\
\midrule\multirow{3}{*}{Classic} & $M = 4$ & 10.71 & (2.43) & 8.63 & (2.28) & 7.96 & (1.93) \\
& $M = 8$ & 7.61 & (3.41) & 5.86 & (3.36) & 5.52 & (2.78) \\
& $M = 12$ & 6.40 & (4.06) & 4.94 & (3.98) & 4.70 & (3.26) \\
\midrule\multirow{3}{*}{Rolling QR} & $M = 4$ & 5.23 & (4.97) & 4.46 & (4.42) & 3.82 & (4.02) \\
& $M = 8$ & 6.15 & (4.23) & 5.59 & (3.52) & 4.91 & (3.13) \\
& $M = 12$ & 6.04 & (4.30) & 5.61 & (3.51) & 5.16 & (2.98) \\
\midrule
\multirow{3}{*}{\shortstack{Extrapolation \\ (Least-sq., $m = 2$)}} & $M = 4$ & 5.14 & (5.06) & 4.49 & (4.38) & 5.49 & (2.80) \\
& $M = 8$ & 5.81 & (4.48) & 6.15 & (3.20) & 6.46 & (2.38) \\
& $M = 12$ & 6.81 & (3.82) & 6.79 & (2.90) & 6.95 & (2.21) \\
\midrule
\multirow{3}{*}{\shortstack{Extrapolation \\ (Least-sq., $m = 3$)}} & $M = 4$ & 5.56 & (4.67) & 4.68 & (4.20) & 3.92 & (3.92) \\
& $M = 8$ & 4.07 & (6.39) & 3.74 & (5.26) & 3.34 & (4.60) \\
& $M = 12$ & 3.76 & (6.91) & 3.55 & (5.55) & 3.54 & (4.34) \\
\midrule
\multirow{3}{*}{\shortstack{Extrapolation \\ (Least-sq., $m = 4$)}} & $M = 4$ & \multicolumn{1}{c}{---} & \multicolumn{1}{c |}{---} & \multicolumn{1}{c}{---} & \multicolumn{1}{c |}{---} & \multicolumn{1}{c}{---} & \multicolumn{1}{c}{---} \\
& $M = 8$ & 5.05 & (5.14) & 4.10 & (4.79) & 3.15 & (4.88) \\
& $M = 12$ & 3.92 & (6.63) & 3.19 & (6.17) & 3.21 & (4.78) \\
\midrule
\multirow{3}{*}{\shortstack{Extrapolation \\ (Sparse, $m = 2$)}} & $M = 4$ & 5.13 & (5.07) & 4.76 & (4.13) & 5.62 & (2.73) \\
& $M = 8$ & 5.23 & (4.97) & 5.85 & (3.36) & 6.40 & (2.40) \\
& $M = 12$ & 6.07 & (4.28) & 6.49 & (3.03) & 6.73 & (2.28) \\
\midrule
\multirow{3}{*}{\shortstack{Extrapolation \\ (Sparse, $m = 3$)}} & $M = 4$ & 5.55 & (4.68) & 4.65 & (4.23) & 3.92 & (3.92) \\
& $M = 8$ & 4.22 & (6.17) & 3.52 & (5.59) & 3.32 & (4.62) \\
& $M = 12$ & 3.80 & (6.84) & 3.37 & (5.84) & 3.43 & (4.47) \\
\midrule
\multirow{3}{*}{\shortstack{Extrapolation \\ (Sparse, $m = 4$)}} & $M = 4$ & \multicolumn{1}{c}{---} & \multicolumn{1}{c |}{---} & \multicolumn{1}{c}{---} & \multicolumn{1}{c |}{---} & \multicolumn{1}{c}{---} & \multicolumn{1}{c}{---} \\
& $M = 8$ & 4.98 & (5.22) & 3.72 & (5.30) & 3.16 & (4.86) \\
& $M = 12$ & 4.42 & (5.89) & 3.25 & (6.06) & 3.21 & (4.79) \\
\bottomrule
\end{tabular}
\caption{Average number of velocity solver PCG iterations at each time step for
the same problem as in Table \ref{TAB:FSRTimes}.  The numbers in parentheses
give the factor by which the various methods reduce the number of iterations
relative to using the solution at the last time step as an initial guess.}
\label{TAB:FSRVelocityIters}
\end{table}

\begin{table}
\centering
\begin{tabular}{c l | r r | r r | r r}
\toprule
\multicolumn{2}{c |}{Method} & \multicolumn{2}{c |}{$\mathrm{Re} = 200$} & \multicolumn{2}{c |}{$\mathrm{Re} = 500$} & \multicolumn{2}{c}{$\mathrm{Re} = 1000$} \\
\midrule
\multicolumn{2}{c |}{Last time step} & 14.77 & (1.00) & 15.19 & (1.00) & 15.60 & (1.00) \\
\midrule\multirow{3}{*}{Classic} & $M = 4$ & 6.57 & (2.25) & 6.91 & (2.20) & 7.28 & (2.14) \\
& $M = 8$ & 4.67 & (3.16) & 4.81 & (3.16) & 5.26 & (2.97) \\
& $M = 12$ & 3.81 & (3.87) & 3.92 & (3.87) & 4.35 & (3.59) \\
\midrule\multirow{3}{*}{Rolling QR} & $M = 4$ & 3.60 & (4.11) & 3.96 & (3.84) & 4.07 & (3.83) \\
& $M = 8$ & 2.31 & (6.40) & 2.36 & (6.44) & 3.03 & (5.16) \\
& $M = 12$ & 1.83 & (8.06) & 2.02 & (7.53) & 2.50 & (6.25) \\
\midrule
\multirow{3}{*}{\shortstack{Extrapolation \\ (Least-sq., $m = 2$)}} & $M = 4$ & 4.07 & (3.63) & 4.05 & (3.75) & 4.08 & (3.83) \\
& $M = 8$ & 2.70 & (5.47) & 2.64 & (5.75) & 2.86 & (5.45) \\
& $M = 12$ & 2.28 & (6.47) & 2.29 & (6.63) & 2.71 & (5.75) \\
\midrule
\multirow{3}{*}{\shortstack{Extrapolation \\ (Least-sq., $m = 3$)}} & $M = 4$ & 6.07 & (2.43) & 6.07 & (2.50) & 6.05 & (2.58) \\
& $M = 8$ & 3.55 & (4.16) & 3.54 & (4.29) & 3.55 & (4.40) \\
& $M = 12$ & 2.71 & (5.46) & 2.65 & (5.74) & 2.63 & (5.94) \\
\midrule
\multirow{3}{*}{\shortstack{Extrapolation \\ (Least-sq., $m = 4$)}} & $M = 4$ & \multicolumn{1}{c}{---} & \multicolumn{1}{c |}{---} & \multicolumn{1}{c}{---} & \multicolumn{1}{c |}{---} & \multicolumn{1}{c}{---} & \multicolumn{1}{c}{---} \\
& $M = 8$ & 4.78 & (3.09) & 4.76 & (3.19) & 4.77 & (3.27) \\
& $M = 12$ & 3.44 & (4.30) & 3.40 & (4.46) & 3.40 & (4.59) \\
\midrule
\multirow{3}{*}{\shortstack{Extrapolation \\ (Sparse, $m = 2$)}} & $M = 4$ & 4.14 & (3.57) & 4.09 & (3.72) & 4.10 & (3.81) \\
& $M = 8$ & 2.74 & (5.40) & 2.67 & (5.69) & 2.82 & (5.52) \\
& $M = 12$ & 2.40 & (6.15) & 2.42 & (6.27) & 2.67 & (5.84) \\
\midrule
\multirow{3}{*}{\shortstack{Extrapolation \\ (Sparse, $m = 3$)}} & $M = 4$ & 6.08 & (2.43) & 6.05 & (2.51) & 6.06 & (2.58) \\
& $M = 8$ & 3.77 & (3.92) & 3.75 & (4.05) & 3.73 & (4.19) \\
& $M = 12$ & 2.87 & (5.14) & 2.87 & (5.29) & 2.89 & (5.40) \\
\midrule
\multirow{3}{*}{\shortstack{Extrapolation \\ (Sparse, $m = 4$)}} & $M = 4$ & \multicolumn{1}{c}{---} & \multicolumn{1}{c |}{---} & \multicolumn{1}{c}{---} & \multicolumn{1}{c |}{---} & \multicolumn{1}{c}{---} & \multicolumn{1}{c}{---} \\
& $M = 8$ & 5.12 & (2.89) & 5.10 & (2.98) & 5.12 & (3.05) \\
& $M = 12$ & 3.68 & (4.01) & 3.68 & (4.13) & 3.69 & (4.23) \\
\bottomrule
\end{tabular}
\caption{Average number of pressure solver PCG iterations at each time step for
the same problem as in Table \ref{TAB:FSRTimes}.  The numbers in parentheses
give the factor by which the various methods reduce the number of iterations
relative to using the solution at the last time step as an initial guess.}
\label{TAB:FSRPressureIters}
\end{table}

Focusing on the pressure solves, Figure \ref{SFIG:FSRItersVsTimeStep1} shows
the number of pressure iterations taken at several time steps during the
$\mathrm{Re} = 500$ run for the various strategies, using a space of dimension
$M = 8$.  The projection methods begin to produce substantially lower iteration
counts versus using the solution at the last time step after about 10 time
steps, while the extrapolation methods take closer to 20 time steps to do the
same.  The plot shows the stark difference between the classic and rolling QR
schemes:  every 8 steps, the iteration count for the classic scheme jumps to
match that for using the previous time step's solution, reflecting the fact
that it periodically discards the history space.  In contrast, the rolling QR
scheme maintains counts that are consistently low.  After about 45 time steps,
the iteration counts for EXTRAP(4,~8) reach the same level as rolling QR.  The
counts for EXTRAP(2,~8) and EXTRAP(3,~8) are somewhat higher but remain on a
downward trend.

Figure \ref{SFIG:FSRItersVsTimeStep2} illustrates the long-term behavior of the
pressure solver iteration counts over the course of the full simulation,
displaying the average at each time step of the iteration counts over all the
time steps taken to that point.  While all of the strategies for selecting
initial guesses result in iteration counts that trend downward with time, the
projection and extrapolation methods yield average iteration counts that are
considerably lower.  After 8700 time steps, pressure solves seeded with the
previous time step's solution as the initial guess take about 15 PCG iterations
on average.  For the classic projection scheme with $M = 8$, the long-term
average is just under 5 iterations per solve.  For the rolling QR scheme with
$M = 8$, it is about 2.5.  EXTRAP(2,~8) and EXTRAP(3,~8) compete closely with
QR(8), with long-term averages of about 2.6 and 3.5, respectively.

\begin{figure}
\centering
\subfloat[]{\label{SFIG:FSRItersVsTimeStep1}\begin{tikzpicture}[scale = 0.7]
  \begin{axis}[
    scale = 1.0,
    samples = 100,
    xmin = 0, xmax = 80,
    ymin = 0, ymax = 30,
    xlabel = {Time step},
    ylabel = {No.\ of pressure PCG iterations},
    legend style = { at = {(0.99, 0.99)}, anchor = north east, nodes = {scale = 0.5} },
    legend cell align = left
  ]

  \addplot[mark = x, thick, cLA] file {figures/data/figItersVsTimeStep_LA.dat};
  \addplot[mark = asterisk, thick, cCL] file {figures/data/figItersVsTimeStep_CL.dat};
  \addplot[mark = *, thick, cQR, mark options = {scale = 0.5}] file {figures/data/figItersVsTimeStep_QR.dat};
  \addplot[mark = o, thick, cE2] file {figures/data/figItersVsTimeStep_E2.dat};
  \addplot[mark = square, thick, cE3] file {figures/data/figItersVsTimeStep_E3.dat};
  \addplot[mark = triangle, thick, cE4] file {figures/data/figItersVsTimeStep_E4.dat};

  \legend{LAST, CLASSIC(8), QR(8), {EXTRAP(2, 8)}, {EXTRAP(3, 8)}, {EXTRAP(4, 8)}};

  \end{axis}
\end{tikzpicture}}
\subfloat[]{\label{SFIG:FSRItersVsTimeStep2}\begin{tikzpicture}[scale = 0.7]
  \begin{axis}[
    scale = 1.0,
    samples = 100,
    xmin = 0, xmax = 8737,
    ymin = 0, ymax = 30,
    xlabel = {Time step},
    ylabel = {Cum.\ avg.\ no.\ of pressure PCG iterations},
    legend style = { at = {(0.99, 0.99)}, anchor = north east, nodes = {scale = 0.5} },
    legend cell align = left
  ]

  \addplot[thick, cLA] file {figures/data/figCumAvgIters_LA.dat};
  \addplot[dashed, thick, cCL] file {figures/data/figCumAvgIters_C4.dat};
  \addplot[thick, cCL] file {figures/data/figCumAvgIters_C8.dat};
  \addplot[dashed, thick, cQR] file {figures/data/figCumAvgIters_Q4.dat};
  \addplot[thick, cQR] file {figures/data/figCumAvgIters_Q8.dat};
  \addplot[dashed, thick, cE2] file {figures/data/figCumAvgIters_E2_4.dat};
  \addplot[thick, cE2] file {figures/data/figCumAvgIters_E2_8.dat};
  \addplot[dashed, thick, cE3] file {figures/data/figCumAvgIters_E3_4.dat};
  \addplot[thick, cE3] file {figures/data/figCumAvgIters_E3_8.dat};

  \legend{LAST, CLASSIC(4), CLASSIC(8), QR(4), QR(8), {EXTRAP(2, 4)}, {EXTRAP(2, 8)}, {EXTRAP(3, 4)}, {EXTRAP(3, 8)}};

  \end{axis}
\end{tikzpicture}}
\caption{(a) Pressure solver PCG iteration counts at the first 80 time steps
for the 3D fence flow test problem at $\mathrm{Re} = 500$ for several initial
guess strategies.  (b) Cumulative average pressure iteration counts over the
entire duration of the same simulation to a final time of $t = 1.00$
(approximately 8700 time steps).}
\label{FIG:FSRItersVsTimeStep}
\end{figure}
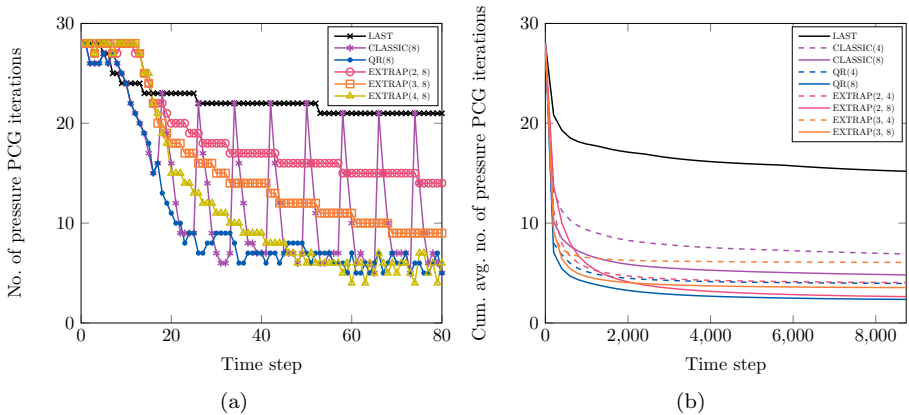

The performance of the sparse extrapolation methods is nearly indistinguishable
from that of the least-squares methods, though the pressure iteration counts
for the sparse methods are, in general, very slightly higher.  This validates
our claim made in Section \ref{SSEC:SparseExtrap} that even though the vector
of extrapolation coefficients for the sparse method does not satisfy a
minimality property like that for the least-squares method, the sparse method
is still a viable approach.

\subsection{Effect of the Stopping Criterion}

The results just reported depend to some extent on the stopping criterion used
for the Krylov solver.  In \texttt{libParanumal}, the default is to stop the
iteration when
\begin{equation}\label{EQN:StopCritINITRESID}
\|r\|_2 < \varepsilon \max\bigl(\|r_0\|_2, 1\bigr),
\end{equation}
where $r$ is the current residual, $r_0$ is the initial residual, and
$\varepsilon$ is a chosen tolerance.  By default, \texttt{libParanumal} uses
$\varepsilon = 10^{-8}$.  This criterion stops iterating when the residual has
been reduced to $\varepsilon$ in a relative sense when the initial residual is
large and in an absolute sense when the initial residual is small.  The results
of the previous section were found using this default criterion.

One alternative to this criterion would be to demand
\begin{equation}\label{EQN:StopCritRHS}
\|r\|_2 < \varepsilon \max\bigl(\|b\|_2, 1\bigr),
\end{equation}
i.e., to ask for a reduction relative not to the size of the initial residual
but to the size of the right-hand side $b$.  The main difference between
\eqref{EQN:StopCritINITRESID} and \eqref{EQN:StopCritRHS} is that if the
initial guess is good ($\|r_0\|_2$ is small), the former criterion will almost
always collapse to $\|r\|_2 < \varepsilon$, while the latter will typically be
$\|r\|_2 < \varepsilon \|b\|_2$, as $\|b\|_2 > 1$ in most cases in the flow
simulations we consider.  That means \eqref{EQN:StopCritRHS} can be a less
stringent requirement than \eqref{EQN:StopCritINITRESID} with a good initial
guess.

We shall not attempt to determine which of the two criteria is ``superior,'' as
this depends on one's accuracy requirements.  Instead, let us examine the
effect of the choice of criterion on the efficacy of the initial guess methods.
Table \ref{TAB:CompareStopCrit} shows the times to solution for our test
problem with $\mathrm{Re} = 500$ for each of the initial guess methods for both
stopping criteria with $\varepsilon = 10^{-8}$.  As the table shows, criterion
\eqref{EQN:StopCritRHS} results in shorter solution times, and the more
sophisticated initial guess methods at less effective at reducing the time
compared to running under criterion \eqref{EQN:StopCritINITRESID}.  Still, even
under \eqref{EQN:StopCritRHS}, the extrapolation methods are able to attain
speedups comparable to (and even slightly better than) those attained by the
projection methods.

Unless otherwise indicated, we will continue to use criterion
\eqref{EQN:StopCritINITRESID} with $\varepsilon  = 10^{-8}$ in what follows.

\begin{table}
\centering
\begin{tabular}{c l | r r | r r}
\toprule
\multicolumn{2}{c |}{Method} & \multicolumn{2}{c |}{Criterion \eqref{EQN:StopCritINITRESID}} & \multicolumn{2}{c}{Criterion \eqref{EQN:StopCritRHS}} \\
\midrule
\multicolumn{2}{c |}{Last time step} & 2586 & (1.00) & 1990 & (1.00) \\
\midrule
\multirow{3}{*}{Classic} & $M = 4$ & 1875 & (1.38) & 1791 & (1.11) \\
& $M = 8$ & 1712 & (1.51) & 1676 & (1.19) \\
& $M = 12$ & 1666 & (1.55) & 1651 & (1.21) \\
\midrule
\multirow{3}{*}{Rolling QR} & $M = 4$ & 1643 & (1.57) & 1526 & (1.30) \\
& $M = 8$ & 1591 & (1.63) & 1570 & (1.27) \\
& $M = 12$ & 1637 & (1.58) & 1658 & (1.20) \\
\midrule
\multirow{3}{*}{\shortstack{Extrapolation \\ (Least-sq., $m = 2$)}} & $M = 4$ & 1519 & (1.70) & 1496 & (1.33) \\
& $M = 8$ & 1434 & (1.80) & 1413 & (1.41) \\
& $M = 12$ & 1421 & (1.82) & 1355 & (1.47) \\
\midrule
\multirow{3}{*}{\shortstack{Extrapolation \\ (Least-sq., $m = 3$)}} & $M = 4$ & 1683 & (1.54) & 1909 & (1.04) \\
& $M = 8$ & 1475 & (1.75) & 1467 & (1.36) \\
& $M = 12$ & 1408 & (1.84) & 1398 & (1.42) \\
\midrule
\multirow{3}{*}{\shortstack{Extrapolation \\ (Least-sq., $m = 4$)}} & $M = 4$ & \multicolumn{1}{c}{---} & \multicolumn{1}{c |}{---} & \multicolumn{1}{c}{---} & \multicolumn{1}{c}{---} \\
& $M = 8$ & 1577 & (1.64) & 1582 & (1.26) \\
& $M = 12$ & 1464 & (1.77) & 1464 & (1.36) \\
\bottomrule
\end{tabular}
\caption{Simulation times (in seconds) for the 3D fence flow test problem at
$\mathrm{Re} = 500$ to a final time of $t = 1.00$ (approximately 8700 time
steps) for each of the initial guess strategies using the two stopping criteria
\eqref{EQN:StopCritINITRESID} and \eqref{EQN:StopCritRHS} with $\varepsilon =
10^{-8}$.  The numbers in parentheses give the speedups attained by the methods
relative to the strategy of using the solution at the last time step.}
\label{TAB:CompareStopCrit}
\end{table}

\subsection{Comparison with Naive Extrapolation}
\label{SSEC:NaiveExtrapExperiments}

It is worth dwelling briefly on how the stabilized extrapolation methods we
have proposed compare with naive extrapolation via Lagrange interpolation as
described in Section \ref{SSEC:LagrangeExtrap}.  As we wrote above, this method
is known to perform poorly as the polynomial degree increases.  This is
confirmed by the numbers in Table \ref{TAB:CompareNaive}, which show a steady
degradation in both solution time and pressure solver iterations when applied
to our test problem with $\mathrm{Re} = 500$.  Comparing with Tables
\ref{TAB:FSRTimes} and \ref{TAB:FSRPressureIters}, we see, for instance, that
EXTRAP(5,~6) is clearly inferior to EXTRAP(2,~4), resulting in more than double
the number of pressure iterations per time step on average despite a 50\%
larger history space.

\begin{table}
\centering
\begin{tabular}{c r r r r}
\toprule
Method & \multicolumn{2}{c}{Solution time} & \multicolumn{2}{c}{Pressure iters.} \\
\midrule
Last time step & 2586 & (1.00) & 15.19 & (1.00) \\
\midrule
EXTRAP(2, 3) & 1589 & (1.63) & 4.98 & (3.05) \\
EXTRAP(3, 4) & 1683 & (1.54) & 6.07 & (2.50) \\
EXTRAP(4, 5) & 1795 & (1.44) & 7.34 & (2.07) \\
EXTRAP(5, 6) & 1905 & (1.36) & 8.62 & (1.76) \\
\bottomrule
\end{tabular}
\caption{Solution times and average pressure solver iteration counts for each
of four extrapolation methods based on naive polynomial interpolation applied
to the fence flow test problem with $\mathrm{Re} = 500$.  The numbers in
parentheses give the improvements attained by the methods relative to the
strategy of using the solution at the last time step.}
\label{TAB:CompareNaive}
\end{table}

The reason for this poor performance is related to the well-understood bad
behavior of polynomial interpolation in \equispaced{} points.  We can quantify
this poor behavior by introducing the notion of a \emph{Lebesgue constant} for
the extrapolation schemes, which is just a suitable norm of the linear operator
that maps the history data to the extrapolated value.  Conceptually, the
Lebesgue constant for a scheme measures the extent to which it amplifies
perturbations to the history data, such as those resulting from round-off error
or the fact that the linear systems are solved only as accurately as demanded
by the Krylov solver tolerance.

In our case, the appropriate definition for the Lebesgue constant $\Lambda$
associated to the extrapolation scheme with coefficient vector $\beta^T$ is
\[
\Lambda = \max_{\|x\|_\infty = 1} |\beta^T x| = \|\beta\|_1,
\]
i.e., the sum of the absolute values of the extrapolation coefficients.  The
following theorem is easy to prove:

\begin{theorem}
The Lebesgue constant for the naive {\rm EXTRAP($M - 1$,~$M$)} extrapolation
scheme is $\Lambda = 2^M - 1$.
\end{theorem}

\begin{proof}
By \eqref{EQN:NaiveExtrapCoeffs},
\[
\|\beta\|_1 = \sum_{i = 1}^M \binom{M}{i - 1} = 2^M - 1.
\]
\end{proof}

This explosive exponential growth of the Lebesgue constant is what ultimately
renders the naive extrapolation scheme useless in practice.  As Table
\ref{TAB:CompareNaive} illustrates, $M$ does not even need to be very large
before the negative impacts are felt.

The Lebesgue constants for the stabilized schemes do not have convenient closed
forms, but we can compute them easily.  Figure \ref{FIG:LebesgueConstants}
plots the Lebesgue constants as a function of $M$ for the native scheme, the
stable scheme EXTRAP($\lfloor\sqrt{M}\rfloor$,~$M$), and the corresponding
sparse stable scheme, where $\lfloor\cdot\rfloor$ denotes the ``floor''
(greatest integer) function.  The Lebesgue constants for the stabilized methods
do grow but at a considerably slower rate than does the Lebesgue constant for
the naive scheme.

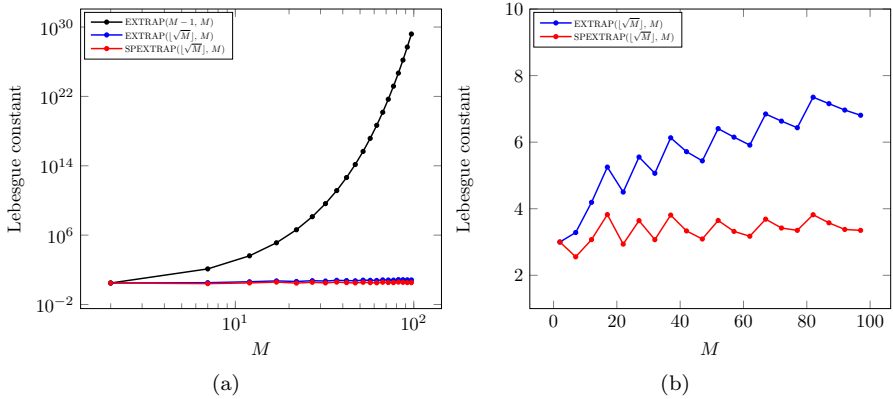
\begin{figure}
\centering
\subfloat[]{\begin{tikzpicture}[scale = 0.7]
  \begin{loglogaxis}[
    scale = 1.0,
    samples = 100,
    xlabel = {$M$},
    ylabel = {Lebesgue constant},
    legend style = { at = {(0.01, 0.99)}, anchor = north west, nodes = {scale = 0.5} },
    legend cell align = left
  ]

  \addplot[mark = *, mark options = {scale = 0.5}, thick, black] file {figures/data/figLebConstNaive.dat};
  \addplot[mark = *, mark options = {scale = 0.5}, thick, blue]  file {figures/data/figLebConstLSQR.dat};
  \addplot[mark = *, mark options = {scale = 0.5}, thick, red]   file {figures/data/figLebConstCPQR.dat};

  \legend{{EXTRAP($M - 1$, $M$)}, {EXTRAP($\lfloor\sqrt{M}\rfloor$, $M$)}, {SPEXTRAP($\lfloor\sqrt{M}\rfloor$, $M$)}}
  \end{loglogaxis}
\end{tikzpicture}}
\subfloat[]{\begin{tikzpicture}[scale = 0.7]
  \begin{axis}[
    scale = 1.0,
    samples = 100,
    ymin = 1, ymax = 10,
    xlabel = {$M$},
    ylabel = {Lebesgue constant},
    legend style = { at = {(0.01, 0.99)}, anchor = north west, nodes = {scale = 0.5} },
    legend cell align = left
  ]

  \addplot[mark = *, mark options = {scale = 0.5}, thick, blue]  file {figures/data/figLebConstLSQR.dat};
  \addplot[mark = *, mark options = {scale = 0.5}, thick, red]   file {figures/data/figLebConstCPQR.dat};

  \legend{{EXTRAP($\lfloor\sqrt{M}\rfloor$, $M$)}, {SPEXTRAP($\lfloor\sqrt{M}\rfloor$, $M$)}}
  \end{axis}
\end{tikzpicture}}
\caption{(a) Lebesgue constants as a function of history space dimension $M$ for
the naive extrapolation scheme, a stable extrapolation scheme, and its sparse
counterpart.  Note the logarithmic scaling of the axes. (b) Same plot on linear axes, omitting the data for the naive extrapolation scheme.}
\label{FIG:LebesgueConstants}
\end{figure}

\subsection{Reynolds Number Dependence}

Table \ref{TAB:FSRTimes} suggests that the effectiveness of the initial guess
methods degrades at least mildly as the Reynolds number increases.  To
investigate this, we ran our solver on our 3D fence flow test problem for
several flows with Reynolds numbers up to $\mathrm{Re} = \text{7,000}$ and
measured the average number of iterations taken by the pressure solver.  The
results are displayed in Figure \ref{FIG:FSRItersVsRe}.  We observe that,
indeed, the iteration counts do rise for both methods as the Reynolds number
increases.

We also ran simulations at the same Reynolds numbers that used a zero initial
guess at all time steps.  The average pressure iteration counts for these
simulations are nearly steady at approximately 24 PCG steps per solve,
irrespective of Reynolds number.  This shows that our preconditioner is
$\mathrm{Re}$-independent (or at most only very weakly $\mathrm{Re}$-dependent)
and hence that $\mathrm{Re}$ dependence of the iteration counts of the initial
guess methods is due to the methods themselves and how they interact with the
equations \eqref{EQN:INS} and our discretization, not our preconditioner.

At $\mathrm{Re} = \text{7,000}$, the effectiveness of the initial guess
strategy of using the previous time step's solution is questionable, as it
saves just 4 iterations per solve on average compared with using a zero initial
guess.  While the projection and extrapolation schemes' effectiveness at
$\mathrm{Re} = \text{7,000}$ is diminished compared with low Reynolds numbers,
they still take fewer than half the number of iterations as a zero initial
guess.

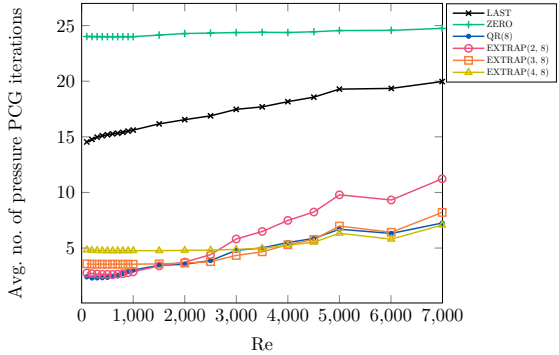
\begin{figure}
\centering
\begin{tikzpicture}[scale = 0.7]
  \begin{axis}[
    scale = 1.0,
    samples = 100,
    xmin = 0, xmax = 7000,
    xlabel = {Re},
    ylabel = {Avg.\ no.\ of pressure PCG iterations},
    legend style = { at = {(1.01, 1)}, anchor = north west, nodes = {scale = 0.5} },
    legend cell align = left
  ]

  \addplot[mark = x, thick, cLA] file {figures/data/figItersVsRe_LA.dat};
  \addplot[mark = +, thick, cZE] file {figures/data/figItersVsRe_ZE.dat};
  \addplot[mark = *, mark options = {scale = 0.5}, thick, cQR] file {figures/data/figItersVsRe_QR.dat};
  \addplot[mark = o, thick, cE2] file {figures/data/figItersVsRe_E2.dat};
  \addplot[mark = square, thick, cE3] file {figures/data/figItersVsRe_E3.dat};
  \addplot[mark = triangle, thick, cE4] file {figures/data/figItersVsRe_E4.dat};

  \legend{LAST, ZERO, QR(8), {EXTRAP(2, 8)}, {EXTRAP(3, 8)}, {EXTRAP(4, 8)}}

  \end{axis}
\end{tikzpicture}
\caption{Average number of pressure solver PCG iterations for the 3D fence flow
test problem, run to a final time of $t = 1.00$ (approximately 8700 time steps)
for flows with Reynolds numbers varying from 200 to 7,000 for several initial
guess strategies.  Legend labels correspond to Table \ref{TAB:DataMoveCounts}.
``ZERO'' gives the results obtained using a zero initial guess for the pressure
solution at each time step.}
\label{FIG:FSRItersVsRe}
\end{figure}

\subsection{Using Different Methods for Velocity and Pressure}

The experiments reported in the previous sections all use the same initial
guess scheme for both the velocity and the pressure solves.  In fact, it may be
advantageous to use different schemes for each.  Table \ref{TAB:MixedResults}
shows the results obtained on our test problem with $\mathrm{Re} = 500$ using
the EXTRAP(3,~8) scheme for velocity and several different schemes for pressure
except for the first row, which gives the results for using the solution at the
last time step as the initial guess for both problems.

Comparing with Table \ref{TAB:FSRTimes}, we see that when the initial guesses
for pressure are found using an extrapolation scheme, using EXTRAP(3,~8) for
velocity leads to essentially the same results as using the same extrapolation
scheme for velocity as for pressure.  On the other hand, using EXTRAP(3,~8) for
velocity and the rolling QR projection scheme for pressure is a marked
improvement over using the rolling QR scheme for both.  The reason is that the
extrapolation scheme produces lower velocity iteration counts, cf. Table
\ref{TAB:FSRVelocityIters}.  Despite this improvement, however, schemes that
use extrapolation for both velocity and pressure still attain slightly better
times.

Table \ref{TAB:MixedResults} shows only a small subset of the possible
combinations.  We leave the determination of the ``best'' combination as a
matter for future work.

\begin{table}
\centering
\begin{tabular}{c l | r r | r r | r r}
\toprule
\multicolumn{2}{c |}{Method} & \multicolumn{2}{c |}{Solution time} & \multicolumn{2}{c |}{Velocity iters.} & \multicolumn{2}{c}{Pressure iters.} \\
\midrule
\multicolumn{2}{c |}{Last time step} & 2586 & (1.00) & 19.68 & (1.00) & 15.19 & (1.00) \\
\midrule
\multirow{3}{*}{Rolling QR} & $M = 4$ & 1540 & (1.68) & 3.83 & (5.14) & 3.93 & (3.86) \\
& $M = 8$ & 1442 & (1.79) & 3.85 & (5.11) & 2.39 & (6.36) \\
& $M = 12$ & 1437 & (1.80) & 3.78 & (5.20) & 2.06 & (7.36) \\
\midrule
\multirow{3}{*}{\shortstack{Extrapolation \\ (Least-sq., $m = 2$)}} & $M = 4$ & 1512 & (1.71) & 3.75 & (5.25) & 4.02 & (3.78) \\
& $M = 8$ & 1408 & (1.84) & 3.83 & (5.14) & 2.68 & (5.68) \\
& $M = 12$ & 1381 & (1.87) & 3.80 & (5.18) & 2.32 & (6.54) \\
\midrule
\multirow{3}{*}{\shortstack{Extrapolation \\ (Least-sq., $m = 3$)}} & $M = 4$ & 1672 & (1.55) & 3.77 & (5.22) & 6.01 & (2.53) \\
& $M = 8$ & 1476 & (1.75) & 3.74 & (5.26) & 3.54 & (4.29) \\
& $M = 12$ & 1406 & (1.84) & 3.77 & (5.21) & 2.64 & (5.76) \\
\midrule
\multirow{3}{*}{\shortstack{Extrapolation \\ (Least-sq., $m = 4$)}} & $M = 4$ & \multicolumn{1}{c}{---} & \multicolumn{1}{c |}{---} & \multicolumn{1}{c}{---} & \multicolumn{1}{c |}{---} & \multicolumn{1}{c}{---} & \multicolumn{1}{c}{---} \\
& $M = 8$ & 1571 & (1.65) & 3.76 & (5.23) & 4.73 & (3.21) \\
& $M = 12$ & 1467 & (1.76) & 3.82 & (5.16) & 3.40 & (4.47) \\
\bottomrule
\end{tabular}
\caption{Simulation times (in seconds) for the 3D fence flow test problem to a
final time of $t = 1.00$ (approximately 8700 time steps) using EXTRAP(3,~8) for
velocity and the listed initial guess strategies for pressure except for the
first row, which gives the results for using the solution at the last time step
as the initial guess for both velocity and pressure.  The numbers in
parentheses give the improvements attained by the methods relative to the first
row.}
\label{TAB:MixedResults}
\end{table}

\section{Conclusion}
\label{SEC:Conclusion}

We have studied the effectiveness of projection and extrapolation methods for
computing initial guesses for the solutions to linear systems that arise within
numerical PDE solvers.  We have proposed stabilized extrapolation techniques
based on least-squares polynomial approximation and polynomial interpolation in
well-chosen (non-\equispaced) points and shown that both they and projection
techniques can be implemented efficiently on GPUs.  In spite of their
simplicity, we have found the proposed extrapolation methods to perform
comparably to---and, in some cases, slightly better than---equivalent
projection methods while requiring less storage, moving less data, and
performing fewer operations that require global communication.  In future work,
we plan to study the impact of these techniques on the scalability of
GPU-accelerated PDE solvers across multiple GPUs and compute nodes.

\section*{Acknowledgments}

We acknowledge Advanced Research Computing at Virginia Tech for providing the
computational resources with which we conducted the bulk of our experiments.

\appendix

\section{Details of Data Operation Counts}
\label{APP:OpCountDetails}

In this section, we provide details of the reasoning used to arrive at the data
movement operation counts in Table \ref{TAB:DataMoveCounts}.  For the
projection methods, the reader may find it helpful to refer to Algorithms
\ref{ALG:FSRClassic} and \ref{ALG:FSRQR}.  All vectors are of length $\Ntotal$.
As discussed above, we assume that $\Ntotal \gg M$ so that terms involving only
$M$ can be neglected.

First, we note the following basic counts:

\begin{itemize}

\item Taking a linear combination of $M$ vectors requires loading $M\Ntotal$
vector entries and storing $\Ntotal$, for a total of $(M + 1)\Ntotal$
operations.

\item If, at the same time, we add that linear combination to a given vector,
then we must also load that vector, bringing the count to $(M + 2)\Ntotal$.

\item Computing the inner products of $M$ vectors with a given vector requires
loading $(M + 1)\Ntotal$ values  and storing the $M$ inner products.
Neglecting the latter, we have $(M + 1)\Ntotal$ operations.

\item Applying the Laplacian operator in our matrix-free implementation
requires loading the input and storing the output ($2\Ntotal$ values moved) and
loading the geometric factors ($7\Ntotal$ values moved in 3D).  In all, this
comes to $9\Ntotal$ data movement operations.

\end{itemize}

\subsection{Forming the Initial Guess}

The projection methods form their initial guesses by taking the inner product
of the new right-hand side with the vectors in the right-hand side history
space.  With a space of dimension $M$, this comes to $(M + 1)\Ntotal$
operations.  They then take the computed inner products and use them for
coefficients in a linear combination of $M$ solution vectors.  That takes
another $(M + 1)\Ntotal$ operations for a total of $2(M + 1)\Ntotal$.  Once the
history space is full, QR($M$) will use this many operations every time, but
CLASSIC($M$) doesn't always use a space of size exactly $M$:  it uses spaces of
size $1$, \ldots, $M$ in a cycle.  Averaging the data movement operations
involved in forming the initial guess over these sizes, we obtain a count $(M +
3)\Ntotal$ for CLASSIC($M$).

EXTRAP($m$,~$M$) forms the initial guess from a linear combination of $M$
vectors for a total operation count of $(M + 1)\Ntotal$.  SPEXTRAP($m$,~$M$)
does the same but uses only $m + 1$ vectors.

\subsection{Updating the History Space}

The projection methods update their spaces using twice-iterated Gram--Schmidt.
One round of Gram--Schmidt first computes the inner products of the $M$ vectors
against the new vector, taking $(M + 1)\Ntotal$ operations.  It then uses those
inner products to subtract a linear combination of the $M$ vectors from the new
vector, taking an additional $(M + 2)\Ntotal$ operations.  The
orthogonalization is only done to the right-hand side space, but the same
linear combination must be carried out with the corresponding vectors in the
solution history space.  That takes another $(M + 2)\Ntotal$ operations, for a
total of $(3M + 5)\Ntotal$.

Twice-iterated Gram--Schmidt therefore takes $(6M + 10)\Ntotal$ operations.
Prior to running Gram--Schmidt, the projection operations apply the operator
once, which takes $9\Ntotal$ operations.  To finish the update, we have to
compute the norm of the newly projected right-hand side vector, which performs
one more load of length $\Ntotal$.  We then scale the new right-hand side and
corresponding solution vectors by this norm and store them in the history
arrays, which takes another $4\Ntotal$ movement operations (load and store two
vectors), ignoring the need to load the norm itself.  Thus, the total amount of
data moved during this part of the update process is $(6M + 24)\Ntotal + 6M$.
To obtain the count for CLASSIC($M$), we simply average this count over the
restart cycle like we did before, giving $(3M + 27)\Ntotal$.

In addition to the twice-iterated Gram--Schmidt update, QR($M$) must do the QR
update at every time step.  This requires loading the $M \times M$ matrix $R$,
which we neglect.  Then, we apply a sequence of $M - 1$ Givens rotations, each
of which operates on a pair of columns from the right-hand side history matrix.
It would seem that each rotation requires $4\Ntotal$ data movement operations:
load the two columns to be processed and then store the results, but by being
keeping the column in common used by two successive rotations in registers, we
can cut this in half to $2\Ntotal$ operations per rotation.  We have to do the
same operations on the solution space, which brings us back to $4\Ntotal$
operations per rotation, giving a total of $4(M - 1)\Ntotal$ operations for the
whole sequence of Givens sweeps.  Adding in the $2\Ntotal$ loads for the
initial columns and the $2\Ntotal$ stores for zeroing out the final columns, we
get a total of $4M\Ntotal$ operations.  Adding the counts for the
twice-iterated Gram--Schmidt and QR update operations together, QR($M$)
requires $(10M + 24)\Ntotal + M^2 + 6M$ data movement operations to update its
solution space.

The extrapolation methods do not require any operations to update the history
space, provided that the linear solver stores the solutions in place in the
history array.

\bibliographystyle{siamplain}
\bibliography{main}

\begin{thebibliography}{10}

\bibitem{gslib}
{\em gslib}, 2020, \url{https://github.com/Nek5000/gslib}.
\newblock Release 1.07.

\bibitem{AS1968}
{\sc M.~Abramowitz and I.~A. Stegun}, {\em Handbook of Mathematical Functions
  With Formulas, Graphs, and Mathematical Tables}, vol.~55 of Applied
  Mathematics Series, U.S.\ Dept.\ of Commerce, National Bureau of Standards,
  1968.
\newblock Seventh printing, with corrections.

\bibitem{AS2012}
{\sc M.~al~Sayed~Ali and M.~Sadkane}, {\em Improved predictor schemes for large
  systems of linear {ODE}s}, Electron. Trans. Numer. Anal., 39 (2012),
  pp.~253--270.

\bibitem{bos2010computing}
{\sc L.~Bos, S.~De~Marchi, A.~Sommariva, and M.~Vianello}, {\em Computing
  multivariate {F}ekete and {L}eja points by numerical linear algebra}, SIAM J.
  Numer. Anal., 48 (2010), pp.~1984--1999,
  \url{https://doi.org/10.1137/090779024}.

\bibitem{BX2009}
{\sc J.~P. Boyd and F.~Xu}, {\em Divergence ({R}unge {P}henomenon) for
  least-squares polynomial approximation on an equispaced grid and
  {M}ock-{C}hebyshev subset interpolation}, Appl. Math. Comput., 210 (2009),
  pp.~158--168, \url{https://doi.org/10.1016/j.amc.2008.12.087}.

\bibitem{BILO1997}
{\sc R.~Brower, T.~Ivanenko, A.~Levi, and K.~Orginos}, {\em Chronological
  inversion method for the {D}irac matrix in hybrid {M}onte {C}arlo}, Nucl.
  Phys. B, 484 (1997), pp.~353--374,
  \url{https://doi.org/10.1016/s0550-3213(96)00579-2}.

\bibitem{ChalmersKarakusAustinSwirydowiczWarburton2020}
{\sc N.~Chalmers, A.~Karakus, A.~P. Austin, K.~Swirydowicz, and T.~Warburton},
  {\em {libParanumal}: a performance portable high-order finite element
  library}, 2020, \url{https://doi.org/10.5281/zenodo.4004744},
  \url{https://github.com/paranumal/libparanumal}.
\newblock Release 0.3.1.

\bibitem{Chr2017}
{\sc N.~Christensen}, {\em Efficient Projection Space Updates for the
  Approximation of Iterative Solutions to Linear Systems with Successive Right
  Hand Sides}, {M.S.}\ thesis, University of Illinois at Urbana-Champaign,
  2017, \url{https://doi.org/2142/99410}.

\bibitem{CWSW2004}
{\sc M.~Clemens, M.~Wilke, R.~Schuhmann, and T.~Weiland}, {\em Subspace
  projection extrapolation scheme for transient field simulations}, IEEE Trans.
  Magnetics, 40 (2004), pp.~934--937,
  \url{https://doi.org/10.1109/tmag.2004.824583}.

\bibitem{CWW2003}
{\sc M.~Clemens, M.~Wilke, and T.~Weiland}, {\em Extrapolation strategies in
  numerical schemes for transient magnetic field simulations}, IEEE Trans.
  Magnetics, 39 (2003), pp.~1171--1174,
  \url{https://doi.org/10.1109/tmag.2003.810523}.

\bibitem{DGKS1976}
{\sc J.~W. Daniel, W.~B. Gragg, L.~Kaufman, and G.~W. Stewart}, {\em
  Reorthogonalization and stable algorithms for updating the {G}ram-{S}chmidt
  {QR} factorization}, Math. Comp., 30 (1976), pp.~772--795,
  \url{https://doi.org/10.1090/S0025-5718-1976-0431641-8}.

\bibitem{DT2018}
{\sc L.~Demanet and A.~Townsend}, {\em Stable extrapolation of analytic
  functions}, Found. Comput. Math., 19 (2018), pp.~297--331,
  \url{https://doi.org/10.1007/s10208-018-9384-1}.

\bibitem{DFM2002}
{\sc M.~O. Deville, P.~F. Fischer, and E.~H. Mund}, {\em High-Order Methods for
  Incompressible Fluid Flow}, Cambridge University Press, Cambridge, UK, 2002.

\bibitem{FLPS2008}
{\sc P.~Fischer, J.~Lottes, D.~Pointer, and A.~Siegel}, {\em Petascale
  algorithms for reactor hydrodynamics}, J. Phys.: Conf. Ser., 125 (2008),
  p.~012076, \url{https://doi.org/10.1088/1742-6596/125/1/012076}.

\bibitem{Fis1998}
{\sc P.~F. Fischer}, {\em Projection techniques for iterative solution of {$Ax
  = b$} with successive right-hand sides}, Comput. Methods Appl. Mech. Eng.,
  163 (1998), pp.~193--204,
  \url{https://doi.org/10.1016/s0045-7825(98)00012-7}.

\bibitem{Gan2015}
{\sc R.~Gandham}, {\em High Performance High-Order Numerical Methods:
  Applications in Ocean Modeling}, {Ph.D.}\ thesis, Rice University, 2015.

\bibitem{GV2013_4e}
{\sc G.~H. Golub and C.~F. Van~Loan}, {\em Matrix Computations}, Johns Hopkins
  University Press, Baltimore, 4th~ed., 2013.

\bibitem{Gre1997}
{\sc A.~Greenbaum}, {\em Iterative Methods for Solving Linear Systems}, SIAM,
  Philadelphia, 1997.

\bibitem{GK2011}
{\sc L.~Grinberg and G.~E. Karniadakis}, {\em Extrapolation-based acceleration
  of iterative solvers: Application to simulation of {3D} flows}, Commun.
  Comput. Phys., 9 (2011), pp.~607--626,
  \url{https://doi.org/10.4208/cicp.301109.080410s}.

\bibitem{HH2005}
{\sc J.~M. Herbert and M.~Head-Gordon}, {\em Accelerated, energy-conserving
  {B}orn--{O}ppenheimer molecular dynamics via {F}ock matrix extrapolation},
  Phys. Chem. Chem. Phys., 7 (2005), p.~3269,
  \url{https://doi.org/10.1039/b509494a}.

\bibitem{HS1952}
{\sc M.~R. Hestenes and E.~Stiefel}, {\em Methods of conjugate gradients for
  solving linear systems}, J. Res. Nat. Bur. Stand., 49 (1952), pp.~409--436,
  \url{https://doi.org/10.6028/jres.049.044}.

\bibitem{Joy1971}
{\sc D.~C. Joyce}, {\em Survey of extrapolation processes in numerical
  analysis}, SIAM Rev., 13 (1971), pp.~435--490,
  \url{https://doi.org/10.1137/1013092}.

\bibitem{KCSW2019}
{\sc A.~Karakus, N.~Chalmers, K.~{\' S}wirydowicz, and T.~Warburton}, {\em A
  {GPU} accelerated discontinuous {G}alerkin incompressible flow solver}, J.
  Comput. Phys., 390 (2019), pp.~380--404,
  \url{https://doi.org/10.1016/j.jcp.2019.04.010}.

\bibitem{KS2005}
{\sc G.~E. Karniadakis and S.~Sherwin}, {\em Spectral/hp Element Methods for
  Computational Fluid Dynamics}, Oxford University Press, Oxford, 2nd~ed.,
  2005.

\bibitem{Par1998}
{\sc B.~N. Parlett}, {\em The Symmetric Eigenvalue Problem}, SIAM,
  Philadelphia, 1998.

\bibitem{PH2018}
{\sc G.~Pitton and L.~Heltai}, {\em Accelerating the iterative solution of
  convection-diffusion problems using singular value decomposition}, Numer.
  Linear Algebra Appl., 26 (2018), pp.~e2211 (1--24),
  \url{https://doi.org/10.1002/nla.2211}.

\bibitem{platte2011impossibility}
{\sc R.~B. Platte, L.~N. Trefethen, and A.~B. Kuijlaars}, {\em Impossibility of
  fast stable approximation of analytic functions from equispaced samples},
  SIAM Rev., 53 (2011), pp.~308--318, \url{https://doi.org/10.1137/090774707}.

\bibitem{PF2004}
{\sc P.~Pulay and G.~Fogarasi}, {\em {F}ock matrix dynamics}, Chem. Phys.
  Lett., 386 (2004), pp.~272--278,
  \url{https://doi.org/10.1016/j.cplett.2004.01.069}.

\bibitem{Rak2007}
{\sc E.~A. Rakhmanov}, {\em Bounds for polynomials with a unit discrete norm},
  Ann. of Math. (2), 165 (2007), pp.~55--88,
  \url{https://doi.org/10.4007/annals.2007.165.55}.

\bibitem{Saa2003}
{\sc Y.~Saad}, {\em Iterative Methods for Sparse Linear Systems}, SIAM,
  Philadelphia, 2nd~ed., 2003.

\bibitem{SS1986}
{\sc Y.~Saad and M.~H. Schultz}, {\em {GMRES}: A generalized minimal residual
  algorithm for solving nonsymmetric linear systems}, SIAM J. Sci. Stat.
  Comput., 7 (1986), pp.~856--869, \url{https://doi.org/10.1137/0907058}.

\bibitem{Sht2015}
{\sc K.~S. Shterev}, {\em Iterative process acceleration of calculation of
  unsteady, viscous, compressible, and heat-conductive gas flows}, Int. J.
  Numer. Meth. Fluids, 77 (2015), pp.~108--122,
  \url{https://doi.org/10.1002/fld.3979}.

\bibitem{SCW2006}
{\sc A.~Skarlatos, M.~Clemens, and T.~Weiland}, {\em Start vector generation
  for implicit newmark time integration of the wave equation}, IEEE Trans.
  Magnetics, 42 (2006), pp.~631--634,
  \url{https://doi.org/10.1109/tmag.2006.872008}.

\bibitem{TZ2002}
{\sc A.~G. Tijhuis and P.~M. Zwamborn}, {\em Marching on in anything: Solving
  electromagnetic field equations with a varying physical parameter}, in
  Ultra-Wideband, Short-Pulse Electromagnetics 5, P.~D. Smith and S.~R. Cloude,
  eds., Kluwer Academic, New York, 2002, pp.~655--662,
  \url{https://doi.org/10.1007/0-306-47948-6_78}.

\bibitem{Tre2013}
{\sc L.~N. Trefethen}, {\em Approximation Theory and Approximation Practice},
  SIAM, Philadelphia, 2013.

\bibitem{Tuf1998}
{\sc H.~M. Tufo~III}, {\em Algorithms for Large-Scale Parallel Simulation of
  Unsteady Incompressible Flows in Three-Dimensional Complex Geometries},
  {Ph.D.}\ thesis, Brown University, 1998.

\end{thebibliography}

\end{document}